\documentclass[11pt]{article}
\RequirePackage[OT1]{fontenc}
\RequirePackage{amsthm,amsmath}
\RequirePackage[numbers]{natbib}
\RequirePackage[colorlinks,citecolor=blue,urlcolor=blue]{hyperref}
\usepackage{mathrsfs}
\usepackage{amsfonts}
\usepackage{color}
\usepackage{pifont}
\usepackage{amssymb}
\usepackage[english]{babel}
\usepackage[T1]{fontenc}
\usepackage{graphicx}
\usepackage{subfigure}
\usepackage{latexsym}
\usepackage{amstext}
\usepackage{wrapfig}
\usepackage{yfonts}
\usepackage{hyperref}
\usepackage{verbatim}

\usepackage{amsthm}
\usepackage{mathrsfs}
\usepackage{dsfont}
\usepackage{graphics}
\usepackage{enumerate}
\usepackage{paralist}
\usepackage{fullpage}
\newtheorem{prop}{Proposition}[section]

\numberwithin{equation}{section}

\newcommand{\beq}{\begin{eqnarray}}
\newcommand{\beqq}{\begin{eqnarray*}}
\newcommand{\eeq}{\end{eqnarray}}
\newcommand{\eeqq}{\end{eqnarray*}}

\usepackage{mathtools}
\usepackage[english]{babel}
\usepackage[utf8]{inputenc}
\usepackage[colorinlistoftodos]{todonotes}
\usepackage{polynom}
\usepackage{dsfont}
\newtheorem{theorem}{Theorem}[section]

\newtheorem{lemma}{Lemma}[section]

\usepackage{fullpage}
\usepackage[T1]{fontenc}
\usepackage{hyperref}
\usepackage{xcolor}
\definecolor{link-color}{rgb}{0.15,0.4,0.15}
\hypersetup{
  colorlinks,
  linkcolor = {link-color}, citecolor = {link-color},
}

\usepackage{graphicx}
\usepackage{hyperref}
\usepackage{amsmath,amsthm,amssymb}
\usepackage{bbm}
\usepackage{tabularx}
\usepackage{tikz}

\renewcommand{\P}{\mathbb{P}}
\newcommand{\E}{\mathbb{E}}

    \def\d{{\textnormal d}}

\newenvironment{eqnarr}{\begin{IEEEeqnarray}{rCl}}{\end{IEEEeqnarray}\ignorespacesafterend}
\renewcommand{\eqref}[1]{\hyperref[#1]{(\ref*{#1})}}

\newcommand*{\norm}[1]{\lVert #1 \rVert}

    \def\beq{\begin{eqnarr}}%arw
    \def\eeq{\end{eqnarr}}%arw
    \def\beqq{\begin{eqnarray*}} %arw
    \def\eeqq{\end{eqnarray*}} %arw

        \def\d{{\rm d}}

    \def\d{{\textnormal d}}
\newtheorem{remark}{Remark}[section]

\newcommand*{\pref}[1]{\hyperref[#1]{(\ref*{#1})}}
\newcommand*{\refpref}[2]{\hyperref[#2]{\ref*{#1}(\ref*{#2})}}

\everymath{\displaystyle}

\makeatletter
\def\namedlabel#1#2{\begingroup
    #2%
    \def\@currentlabel{#2}%
    \phantomsection\label{#1}\endgroup
}
\makeatother

  \newcommand{\D}{{\rm d}}

\usepackage{cancel} % to cancel out terms in a maths equation
\usepackage{soul} %underline with \ul and cancel out with \st, with nice behaviour when return to the line occur
\newcommand{\U}{\texttt U}

\newcommand{\bS}{\texttt S}

\newcommand{\bF}{\texttt F}

\newcommand{\up}{\upsilon}

\def\d{{\textnormal d}}

\newcommand{\G}{\mathscr G}

\newcommand{\bP}{\mathtt{P}}

\renewcommand{\d}{{\rm d}}
\renewcommand{\P}{\mathbb{P}}

\newcommand{\VV}{\mathsf{V}\!\!\!\mathsf{V}}

\newcommand{\V}{\mathscr V}

\definecolor{amethyst}{rgb}{0.6, 0.4, 0.8}
\definecolor{applegreen}{rgb}{0.55, 0.71, 0.0}
\definecolor{aqua}{rgb}{0.0, 1.0, 1.0}
\definecolor{asparagus}{rgb}{0.53, 0.66, 0.42}
\definecolor{amber(sae/ece)}{rgb}{1.0, 0.49, 0.0}
 	\definecolor{armygreen}{rgb}{0.29, 0.33, 0.13}
	\definecolor{shitbrown}{rgb}{0.43, 0.21, 0.1}
	\definecolor{brightpink}{rgb}{1.0, 0.0, 0.5}
	\definecolor{brightube}{rgb}{0.82, 0.62, 0.91}
	 	\definecolor{byzantine}{rgb}{0.74, 0.2, 0.64}
		\definecolor{chartreuse(web)}{rgb}{0.5, 1.0, 0.0}

\title{Yaglom limit for  critical neutron transport}

\author{S. C. Harris\thanks{Department of Statistics,
University of Auckland,
Private Bag 92019,
Auckland 1142,
New Zealand.
Email: 
\texttt{simon.harris@auckland.ac.uk}
}, 
\ 
E. Horton\thanks{INRIA,
Bordeaux Research Centre,
33405 Talence,
France.
Email: 
\texttt{emma.horton@inria.fr}
}, 
\
A. E. Kyprianou\thanks{Department of Mathematical Sciences, University of Bath, Claverton Down, Bath, BA2 7AY, UK. Email: \texttt{a.kyprianou@bath.ac.uk}}, \
 and. 
 M. Wang\thanks{School of Mathematical,
 and Physical Sciences,
University of Sussex,
Sussex House, Falmer,
Brighton, BN1 9RH.
UK.
Email: 
\texttt{wang@sussex.ac.uk}
}
}
\begin{document}

\maketitle

\begin{abstract}
We consider the classical Yaglom limit theorem for a branching Markov process $X = (X_t, t \ge 0)$, with non-local branching mechanism in the setting that the mean semigroup is critical, i.e. its leading eigenvalue is zero. In particular, we show that there exists a constant $c(f)$ such that
\[
  {\rm Law}\left(\frac{\langle f, X_t\rangle}{t} \bigg| \langle 1, X_t\rangle > 0 \right) \to {\mathbf e}_{c(f)}, \qquad t \to \infty,
\]
where ${\mathbf e}_{c(f)}$ is an exponential random variable with rate $c(f)$ and the convergence is in distribution. 
As part of the proof, we also show that the probability of survival decays inversely proportionally to time. Although Yaglom limit theorems have recently been handled in the setting of branching Brownian motion in a bounded domain and superprocesses, \cite{Ellen, Yanxia}, these results do not allow for non-local branching, which complicates the analysis. 
Our approach and the main novelty of this work is based around a precise result for the scaled asymptotics for the $k$-th martingale moments of $X$ (rather than the Yaglom limit itself). We then illustrate our results in the setting of neutron transport, for which the non-locality is essential, complementing recent developments in this domain \cite{SNTE, SNTEII, SNTEIII, MCNTE, MultiNTE}.
\medskip

\medskip

\end{abstract}
%\tableofcontents

\section{Introduction and main results}\label{intro} 

In this article, we study the long-term behaviour of a class of critical spatial Markov branching processes (MBP) on a general state space $E$, with non-local branching mechanism and probabilities $\P_{\delta_x}$ when issued from a single particle at $x \in E$. {Roughly speaking, what we mean by critical here is that the leading eigenvalue of the 
linear semigroup of the MBP is zero.} Our objective in this article is to show that, like all other critical branching processes with finite variance offspring distributions, MBPs also respect a so-called Yaglom limit. The latter can be thought of as consisting  of two components, stated precisely in Theorems \ref{lineardecaygen} and \ref{Ygen} below, which say that for a MBP $(X_t, t \ge 0)$, for each $x \in E$ and for each bounded measurable function $f$ on $E$, there exist constants $c_1(x), c(f) \in (0, \infty)$ such that
\begin{equation}
\textstyle \lim_{t\to\infty}t \mathbb{P}_{\delta_{x}}(\langle 1, X_t \rangle > 0) = c_1(x),
\label{result1}
\end{equation}
and
\begin{equation}
\text{Law}\left(\left.\frac{ \langle f, X_t \rangle}{t}\right| \langle 1, X_t\rangle> 0\right) \text{ converges  weakly to } \mathbf{e}_{c(f)},
\label{result2}
\end{equation}
as $t\to\infty$ where $\mathbf{e}_{c(f)}$ is an exponentially distributed random variable with rate $c(f)$, or equivalently with mean $1/c(f)$.

\medskip

Yaglom's limit theorem is very classical for Galton-Watson processes, cf. \cite{AthreyaNey},  and a cornerstone of the theory of branching processes. Remarkably, it is only recently that Yaglom limit theorems have found their way to the literature for spatial branching processes, for example \cite{Ellen} in the setting of branching Brownian motion and \cite{Yanxia, RSZ, RSS2} in the setting of superprocesses. Unfortunately none of the approaches taken there can be transferred to the current setting on account of the fact that there is non-local branching, which complicates calculations significantly. A large component of our proof of \eqref{result1} and \eqref{result2} (stated more precisely in Theorems \ref{lineardecaygen} and \ref{Ygen} below) appeals to a new combinatorial methodology which is necessary on account of the spatial inhomogeneity and non-local nature of the branching mechanism. Specifically, our new approach and the main novelty of this work is based around an inductive argument, which describes the convergence of all of the associated martingale moments of the MBP. For the proof of the latter, we make use of the so-called spine decomposition, in order to extract the leading order terms using a combinatorial decomposition.

%Interestingly, there has been previous work on the Yaglom limit in the setting of NBPs dating back to the 1970s, cf. \cite{MWY}. However, this paper has the advantage of an isotropic assumption for scattering and the point process $\mathcal{Z}$, as well as homogeneous cross-section, which allows the system to be treated as if it has local branching. More precisely, $\mathcal{Z}$ has an independent number of points, each of which are uniformly distributed on $V$, which is otherwise taken to be $\mathbb{S}^2$. In the sequel we will refer to this as the {\it isotropic case}.

\medskip

%The situation we consider here is significantly more complicated than the isotropic case. The anisotropic behaviour of the fission means that we cannot treat branching as a local phenomenon as in \cite{MWY}. 

%In addition to this, we provide a new asymptotic analysis of a non-linear form of the NTE which does not currently exist in the literature. The new analysis we present here offers a degree of robustness in that it can be developed further, offering applicability to general branching Markov processes. Indeed, we justify this claim by proving our results for a general class of  branching Markov processes with non-local branching mechanisms, which we define in Section \ref{sec:MBP}.

Let us now formally define our class of MBPs. Let $E$ be a Lusin space. We will write $L_\infty^+(E)$ for the space of non-negative measurable functions, $f$, on $E$ such that  ${\textstyle\Vert f\Vert := \sup_{x \in E}f(x) < \infty}$ and $L_{\infty, 1}^+(E)$ for the set of functions $f \in L_\infty^+(E)$ such that $\Vert f\Vert \le 1$.
A  $(\bP, \Phi)$-MBP, $(X_t, t\geq0)$, is an atomic measure-valued stochastic process in which particles move around independently in $E$ according to a copy of the Markov process associated to $\bP = (\bP_t, t \ge 0)$. When a particle is positioned at $x\in E$, at the instantaneous rate $\gamma(x) \ge 0$, the process will branch and a random number of  offspring, say $N$, are thrown out in positions, say $x_1,\cdots, x_N$ in $E$, according to the law $\mathcal{P}_x$. Such branching events are described by the branching mechanism $\Phi$, defined via
\[
\Phi(x, f) = \gamma(x)\mathcal{E}_x\left(\prod_{i = 1}^N f(x_i) - f(x)\right), \qquad f\in L^+_{\infty, 1}(E), x\in E.
\]

 It is well known that if the configuration of particles at time t is denoted by ${x_1(t), . . . , x_{N_t}(t)}$, where $N_t$ denotes the number of particles alive at time $t \ge 0$, the MBP can be described via the co-ordinate process in the space of atomic measures on $E$, denoted by $\mathcal{M}(E)$, given by
\begin{equation}
X_t(A) = \sum_{i = 1}^{N_t} \delta_{x_i(t)}(A), \quad t \ge 0, A \in \mathcal{B}(E),
\label{empdist}
\end{equation}
where $\delta$ is the Dirac measure, defined on $\mathcal{B}(E)$, the Borel subsets of $E$. 
For convenience, we will write 
\begin{equation}
\mathcal{Z}(A) = \sum_{i =1}^N\delta_{x_i}(A), \qquad A \in\mathcal{B}(E).
\label{PPgen}
\end{equation}

\medskip

{\color{black}We will tend to work with the process $X$ by describing the behaviour of $\langle f, X_t\rangle$, $t\geq 0$, for 
$f\in L^+_\infty(E)$, where $\textstyle \langle f, X_t\rangle : =  \int_{E}f(x)X(\d x)= \sum_{i = 1}^{N_t} f(x_i(t))$. Having introduced this inner product notation, we note that we will immediately abuse it by writing e.g. $\langle f, g\rangle$ to mean $\textstyle \int_{E}f(x)g(x)\d x$, as well as $\textstyle\langle g, \mathcal{Z}\rangle$ to mean $\textstyle\int_E g(x)\mathcal{Z}(\d x)=\sum_{i=1}^N g(x_i)$.

\medskip

We do not need $\bP$ to have the Feller property, and it is not necessary that $\bP$ is conservative. That said, if so desired, we can append  a cemetery state $\{\dagger\}$ to $ E$, which is to be treated as an absorbing state, and regard $\bP$ as conservative on the extended space $E\cup\{\dagger\}$ (which can also be treated as a Lusin space) and $\Phi(\{\dagger\},\cdot)\equiv 0$ (i.e. no branching activity on the cemetery state).  We will write $\mathbb{P}_{\mu}$ for the law of $X$ with initial configuration $\mu \in \mathcal{M}(E)$, and $\mathbb{E}_{\delta_x}$ for the corresponding expectation operator.
%As such, for $f\in L_\infty^+(E)$, we extend its definition to $E\cup\{\dagger\}$ by taking  $f(\dagger): = 0$.
%in which case
%\begin{equation}
%\bP_t[f](x)=\mathbf{E}_x[f(\xi_t)] = \mathbf{E}_x[f(\xi_t)\mathbf{1}_{(t<\tk)}] ,\qquad x\in E, f\in L_\infty^{+}(E),
%\label{PtoP}
%\end{equation}
%where  the process $\xi$, with probabilities $(\mathbf{P}_x,x\in E)$, is the Markov process on $E\cup\{\dagger\}$ with lifetime $\tk =\inf\{t>0: \xi_t =\dagger\}$.

\medskip

{\color{black}To emphasise the generality of our model, we provide some examples of non-local branching processes in the literature that are encompassed by our setting. The reader is also referred to \cite{INW1, INW2, INW3} for an introduction to general branching Markov processes. 

\medskip

\noindent{\bf Neutron branching process:} The neutron branching process (NBP), as introduced in~\cite{MultiNTE, SNTE}, is our first example. Neutrons evolve in the configuration space $E = D \times V$, where $D \subset \mathbb{R}^3$ is a bounded, open set denoting the set of particle locations and $V:= \{\up \in \mathbb{R}^3 : \mathtt{v}_{\min} \le |\up| \le \mathtt{v}_{\max}\}$ with $0 < \mathtt{v}_{\min} \le \mathtt{v}_{\max} < \infty$, denotes the set of velocities. 
From an initial space-velocity configuration $(r, \up)$, particles move according to piecewise deterministic Markov processes characterised by $\sigma_\mathtt{s}\pi_\mathtt{s}$, where $\sigma_\mathtt{s}(r, \up)$ denotes the rate at which particles change velocity (also called {\it scattering} events) at $(r, \up)$, and $\pi_\mathtt{s}(r, \up, \up')\d\up'$ denotes the probability that such a scattering event results in a new outgoing velocity $\up'$. When at $(r, \up) \in D \times V$, at rate $\sigma_\mathtt{f}(r, \up)$, a branching (or {\it fission}) event occurs, resulting in the release of several new neutrons with configurations $(r, \up_1), \dots, (r, \up_N)$, say. Thus, the non-locality appears in the velocity variable. The non-linear branching mechanism and associated equation was given in \cite{SNTEII}. 
{The NTE formed our original motivation for these results, since it is the critical regime in which nuclear reactors are operated and thus a more complete picture of the behaviour of this setting can be obtained from Theorems~\ref{lineardecaygen} and~\ref{Ygen} than results of the form (G2). Interestingly, there has been previous work on the Yaglom limit in the setting of NBPs dating back to the 1970s, cf.~\cite{MWY}, however, their methods relied on some rather restrictive assumptions on the NBP. We will discuss this further in Section~\ref{NTE}}.

\medskip

\noindent{\bf Branching L\'evy processes:} Our second example is a one-dimensional branching L\'evy processes as defined in \cite{branchlevy}. This process is defined via a triple, $(\Lambda, L, \Sigma)$, where a particle moves according to a L\'evy process, $L$, and at the random time $\Lambda$, which is exponentially distributed with parameter $\gamma$, the particle dies and instantaneously scatters a random number of independent copies of $X$ according to the point process $\textstyle \Sigma(\d y) = \sum_i \delta_{x_i} (\d y)%\{\zeta_i : i = 1, \dots, \Sigma(\mathbb{R})\}
$ relative to its point of death. 
\medskip

\noindent{\bf  Uchiyama process:} A relative of the previous model is the one given in \cite{Uchiyama}. Particles live in $\mathbb{R}^d$ and do not move, but after an independent and exponentially distributed random time they give birth at a constant rate $\gamma$ to offspring according to an independent copy of a point process in $\mathbb{R}^d$ centred at the parent's position. In essence this is a continuous-time $d$-dimensional branching random walk.

\medskip

\noindent{\bf Biggins process:} For interest,  we mention the very  general branching  process given in \cite{BigginsBRW}, also called the Biggins process in \cite{AEKThesis}. This process is similarly described by a triple $(Z, M, \chi)$, where $Z$ describes the reproduction, $M$ describes the movement, and $\chi$ is a measure of the importance of an individual. More precisely, a particle moves according to a stochastic process, $(M(a),a \ge 0)$, where $a$ is the age of the particle. At a random time, the parent particle gives birth to a random number of individuals according to the point process $Z$ on $\mathbb{R} \times \mathbb{R}_+$, where the first coordinate describes the spatial displacement of the child from the parent's birth position, and the second coordinate gives the parent's age at the time of that child's birth. Finally, $\chi$, is another random process that can be thought of as the importance of the particle in the population as it ages. This model takes  features of branching Markov processes, continuous-time branching random walks and Crump-Mode-Jagers processes. Although this process is even more general than the ones we consider in this paper, it is worth remarking that it is likely that the techniques employed here are transferable to the setting of Biggins processes, modulo some technical innovation.}
\medskip

Returning to the $(\bP,\Phi)$-branching process, recall that $(N_t, t\geq 0)$ denotes the process giving the number of particles alive at each time $t \ge 0$ and $\{x_i(t) : i = 1, \dots N_t\}$ gives their configurations. Then, the linear semigroup associated with $X$ is given by
 \begin{equation}
   \psi_t[f](x) \coloneqq \E_{\delta_x}\left[ \langle f, X_t\rangle \right], \qquad t \ge 0, \, x \in E, \, f \in L_\infty^+(E).
   \label{linearsgp}
 \end{equation}

We now introduce our first assumption.

\medskip

{\bf \begin{itemize}
\item[(G1):]
  The branching rate, $\gamma(x)$, and the mean number of offspring, 
  \[m(x) = \mathcal{E}_x[\langle 1, \mathcal{Z}\rangle],\] 
  are uniformly bounded above.
\end{itemize}
}

\medskip

We now consider the long-term average behaviour of the process $X$.
%In the special case of the NTE, Theorem \ref{CVtheorem} provides us with the existence of the leading eigenvalue $\lambda_* \in \mathbb R$ and corresponding eigenfunctions, $\varphi$ and $\tilde\varphi$, which describe the leading order behaviour of the linear semigroup. 
%We will soon see that this type of result is fundamental to analysing the martingale moments of the MBP, thus we introduce the following assumption.

\medskip
{\bf \begin{itemize}
\item[(G2):] There exist $\varphi \in L_\infty^+(E)$ and a finite measure $ \tilde\varphi $ on $E$ 
such that $\varphi$ is uniformly bounded away from zero on each compactly embedded subset of $E$ and for all $g\in L^+_{\infty}(E)$,
\begin{equation}
\langle\tilde\varphi, \psi_t[g]\rangle = \langle\tilde\varphi,g\rangle\quad  \text{and }\quad 
\psi_t[\varphi] = \varphi
 \quad \text{ for all } t\ge 0.
\label{earlyleftandright}
\end{equation}
Moreover, there exists $\varepsilon>0$ such that %for all $g\in L^+_\infty(D\times V)$,
\begin{equation}
\sup_{g\in L^+_{\infty}(E)}  \left\|{\varphi}^{-1}{\psi_t[g]}-\langle\tilde\varphi, g\rangle\right\|_\infty = O({\rm e}^{-\varepsilon t}) \quad\text{ for all $t \ge 0$.}
\label{earlyspectralexpsgp}
\end{equation}
\end{itemize}}

\begin{remark}
\rm
It is usual to normalise $\tilde\varphi$ and $\varphi$ such that $\langle\tilde\varphi, \varphi\rangle $ = 1, which we will assume is the case throughout.\hfill$\diamond$
\end{remark}

\begin{remark}[Criticality]\label{criticalityremark}\rm
Note that assumption (G2) entails that the MBP $X$ is what we would call a `critical' setting. {Indeed, for a lot of literature surrounding spatial branching processes, there has been emphasis on results for which an underlying assumption of exponential ergodic growth in the first moment is present; see e.g. \cite{RSZ, EHK, CRY, PY, SNTE, SNTEII}. In other words, this means that the mean semigroup exhibits a Perron-Frobenius type asymptotic of the form
\begin{equation}
\psi_t[f](x)\sim {\rm e}^{\lambda t} \varphi(x)\langle f, \tilde\varphi\rangle,\qquad x\in E,
\label{ergodicspine}
\end{equation}
as $t\to\infty$,
where $\lambda$ is the lead eigenvalue of the mean semigroup $(\psi_t, t\geq0)$ and $\varphi$, $\tilde\varphi$ are the associated right eigenfunction and left eigenmeasure respectively, meaning $\langle\mu,  \psi_t[\varphi] \rangle = {\rm e}^{\lambda t}\langle \mu, \varphi \rangle$ for $t\geq0$, $\mu\in \mathcal{M}(E)$, and $\langle\psi_t[f], \tilde\varphi\rangle = {\rm e}^{\lambda t} \langle \tilde\varphi, f\rangle$, for $f\in L_\infty^+(E)$, $t\geq0$, respectively. 

\medskip

{
The eigenvalue $\lambda$ now gives us a sense of criticality in the spirit of the same notion for Galton-Watson processes. That is to say: The process grows on average if $\lambda>0$ and is referred to as supercritical; decays on average if $\lambda <0$ and referred to as subcritical; maintains a stabilising average if $\lambda = 0$ and is referred to as critical.
}
\medskip

In this sense we see that (G2) deals with critical processes, i.e. $\lambda = 0$.} Although (G2) may seem like a strong assumption, it can be shown to hold for a large class of MBPs. 
%{\color{red}For the examples we have listed above, in order to satisfy (G2), it would likely be necessary to additionally ensure that particles are killed on exiting a compact spatial domain, similarly to the model in \cite{Ellen}. } 
We refer the reader to \cite{CV, CVecp, SNTE} and the references therein for various assumptions that imply (G2), as well as several examples where one can prove such results. \hfill$\diamond$
\end{remark}

We will now introduce two key quantities that will appear in our main results. We first define the operator
\begin{equation}
  {\V}[f, g](x) = \mathcal{E}_x\Bigg[\sum_{{\substack{i, j = 1 \\ i \neq j}}}^N f(x_i)g(x_j) \Bigg],
  \qquad f, g\in L^{+}_{\infty}(E), \, x \in E.
\label{2point}
\end{equation}
We will often abuse notation and write ${\V}[g](x)$ instead of ${\V}[g,g](x)$ for $g\in L^{+}_{\infty}(E)$ and $x \in E$. Secondly, we define the constant
\begin{equation}
  \Sigma = \langle \tilde{\varphi}, \gamma {\V}[\varphi]\rangle.
  \label{sigma}
\end{equation}

\medskip

With this notation in mind, we introduce our first main result, which describes the limiting behaviour of all of the moments of the MBP and also drives the first step of the Yaglom limit.

\begin{theorem}[Asymptotic martingale moments]\label{genmoments}
Suppose that (G1) and (G2) hold. Further, suppose that for some (fixed) $k\geq 1$,
\begin{equation}
\sup_{x \in E}\mathcal{E}_x\left[\langle 1, \mathcal{Z}\rangle^{k+1}\right]<\infty.
\label{genmomentcondition}
\end{equation}
Then
for all  $j\leq k$, 
\begin{equation}
\sup_{x\in E}\mathbb{E}_{\delta_x}[\langle \varphi, X_t\rangle^{j + 1}]<\infty, \qquad t\geq 0,
\label{boundedmoments}
\end{equation}
and
%\begin{equation}
%\label{A}
%\lim_{t\to\infty} \sup_{x \in E}\left|
%\frac{1}{t^{j}}\mathbb{E}^\varphi_{\delta_{x}}\left[\langle\varphi, X_t \rangle^{j} \right] -(j+1)! \left(\frac{\Sigma}{2}\right)^{j}\right| =  0.
%\end{equation}
%Equivalently, for $j\leq k$, under $\mathbb{P}$ we have
\begin{equation}
\label{A}
\lim_{t\to\infty} \sup_{x \in E}\left|
\frac{1}{\varphi(x)t^{j}}\mathbb{E}_{\delta_{x}}\left[\langle\varphi, X_t \rangle^{j+1} \right] -(j+1)! \left(\frac{\Sigma}{2}\right)^{j}\right| =  0.
\end{equation}
\end{theorem}

%In Sections \ref{generalsetup} we switch from the setting of the NBP to a general branching process with non-local branching mechanism and introduce a technical result concerning ergodic limits of certain functionals of the NBP. Finally, Sections \ref{Sectstep1}-\ref{Sectstep3} will be devoted to the proofs of each of these three steps.
\medskip

In order to state our other main results, we set $\zeta := \inf\{t \ge 0 : \langle 1, X_t \rangle = 0\}$ to be the extinction time of the MBP and further impose the following assumptions.
\medskip

{\bf 
\begin{itemize}
\itemsep1em

\item[(G3):] For all $x \in E$, the process becomes extinct $\P_{\delta_x}$-almost surely, i.e. 
\[
\mathbb{P}_{\delta_x}(\zeta < \infty) = 1.
\]
\item[(G4):] The number of offspring produced at a branching event is bounded above by a constant, $N_{\max}$.

\item[(G5):] There exists a constant $C > 0$ such that for all $g \in L^+_\infty(E)$, 
    \[
      \langle \tilde\varphi, \gamma{{\V}}[g]\rangle \ge C\langle \tilde\varphi, {g}\rangle^2.
    \]
\end{itemize}
}

Let us say a few words about these extra assumptions. While assumption (G2) entails that $\psi_t[g](x) \sim \langle \tilde\varphi, g \rangle \varphi(x)$ as $t \to \infty$, it is folklore that for critical branching processes, the process becomes extinct almost surely, as in (G3). {However, because of general spatial considerations, proving this in the general case is rather challenging. Indeed, it is possible to adapt the arguments of the proof of this result given in~\cite{SNTE} to yield almost sure extinction for {\it almost every} $x \in E$, however one can always construct examples where the process is supercritical at a single point, for example.} Thus we include it as an assumption. We refer the {reader to \cite{SNTE, Fish1, Fish2, Fish3, grey, sevast, Hbook} for sufficient conditions and proofs of this result in various settings (which e.g. include NBPs, branching diffusions and multi-type continuous-time Galton-Watson processes). }

\medskip

{Assumption (G4) is necessary for the proof of the asymptotic of the survival probability in Theorem~\ref{lineardecaygen} below. However, we emphasise that, if we assume that the conclusion of Theorem~\ref{lineardecaygen} holds, then (G4) would not be needed for the proof of Theorem \ref{Ygen}; in this case we would require that all moments of the offspring distribution exist.}

\medskip

{Assumption (G5) can be thought of as an irreducibility type condition on the branching operator, arguably playing an analogous  role to e.g. the one uniform ellipticity plays for  diffusion operators. Heuristically speaking, it ensures that the MBP is able to  explore the state space. We will later apply this condition when $g$ is the survival probability. In this case, the right-hand side can be seen as the survival probability of two independent particles started in the stationary distribution $\tilde\varphi$. Hence, (G5) implies that if there is a branching event (in the stationary distribution), the process should have a better chance of survival than producing two particles {\it independently} according to $\tilde\varphi$.}

\medskip

%Again, we abuse notation and define the extinction time $\zeta \coloneqq \inf\{t \ge 0 : \langle 1, X_t \rangle =0\}$. Then, in the setting of the $(\bP, \Phi)$-branching process, Theorems \ref{lineardecay} and \ref{Y} take the following forms.
%
%\medskip

This leads us to our main results regarding the asymptotic behaviour of the survival probability and the process conditioned to survive.

\begin{theorem}[Kolmogorov limit]\label{lineardecaygen}
Suppose that assumptions (G1), (G2), (G3), (G4) and (G5) hold. Then, for all $x \in E$,
\[
\lim_{t\to\infty} t\mathbb{P}_{\delta_{x}}(\zeta>t) = \frac{2\varphi(x)}{\Sigma}.
\]
\end{theorem}

\begin{theorem}[Yaglom limit]\label{Ygen} 
Suppose that assumptions (G1), (G2), (G3), (G4) and (G5) hold. For $f\in L^+_\infty(E)$ and for all $x \in E$, 
\begin{equation}
\lim_{t\to\infty}\mathbb{E}_{\delta_{x}}\left[\left. \exp\left( -\theta \frac{ \langle f, X_t \rangle}{t}\right)\right|\zeta>t\right] = \frac{1}{1 + \langle \tilde\varphi, f\rangle\Sigma  \theta/2}.
\label{YaglomGen}
\end{equation}
\end{theorem}

{As alluded to in \eqref{result2}, another way of reading Theorem \ref{Ygen} is to define $\mathbb{P}^t_{\delta_{x}}(\cdot) = \mathbb{P}_{\delta_{(r,\up)}}(\cdot\, | \,\zeta>t)$, in which case $(t^{-1}\langle f, X_t\rangle , \mathbb{P}^t_{\delta_{(r,\up)}})$ as $t\to\infty$ converges weakly to an exponential random variable with rate $2/\langle \tilde\varphi, f\rangle\Sigma$.

\bigskip

{\color{black}There does not appear to be a comparable result to Theorem \ref{genmoments} for general branching processes in the literature, least of all for processes with non-local branching mechanisms. However, we will make a few remarks below concerning the only branching Markov process models for which we are aware that Theorem \ref{lineardecaygen} and \ref{Ygen} have been proved. 

\medskip

\noindent{\bf Isotropic neutron transport:} First, let us consider \cite{MWY}. Here the authors proved Theorem \ref{lineardecaygen} under the assumption of isotropic fission and scattering, and where the fission rate and yield are independent of the position and velocity of the colliding neutron. In particular, recalling the NBP example discussed previously, we have $\sigma_\mathtt{f}(r, \up) = \sigma$, for some constant $\sigma$ > 0, and further, the law of the offspring $\mathcal{P}$ does not depend on $(r, \up)$. 

\medskip

\noindent {\bf Branching Brownian motion:} Secondly, we consider the model given in \cite{Ellen}. Here, the author considers a branching Brownian motion on a compact domain $D$. More precisely, particles move according to a Brownian motion in $D$ and at rate $\gamma$, a branching event occurs, producing some number of offspring at the site of the branching event. If a particle hits the boundary of $D$, it is killed. Again, in this case, the branching rate $\gamma$ and the offspring distribution are spatially independent. {Furthermore}, in this setting, the left and right eigenfunctions are equal. Hence, the constant $\Sigma$ reads
\[
  \Sigma = \gamma\mathcal{E}[N(N-1)]\int_D \varphi(x)^3\d x,
\]
as given in \cite[Theorem 1.4]{Ellen}.
\medskip

In the spirit of the branching Brownian motion model described above, Branching L\'evy processes, Uchiyama processes and Biggins processes should, in principal, be candidates from which one can similarly develop a Yaglom limit result by killing particles when they exit a compact domain in such a way that the system is critical. {Generally speaking, killing in such a manner, is a way to open the door to  satisfying the required condition (G2).}
}

\medskip

\noindent{\bf Multi-type continuous-time Galton Watson process:} Technically speaking, a multi-type continuous-time Galton Watson process is a spatial branching particle system in which $E$ is countable. Parents do not move during their lifetime, which is independent and exponentially distributed depending on the particle's state say $q_i$, $i\in E$,  and reproduce on their death, sending their offspring both to their current site as well as other sites in $E$ in such a way that the distribution of the offspring depends on the state of the parent. 
%This model is similar in spirit to Uchiyama's model, albeit that  there is  offspring dependency on the parent's state and the state space is less general. 
In the setting that $E$ is finite {and the mean  offspring matrix satisfies an irreducibility} condition, the Yaglom limit theorem is  a classical result, see e.g. Chapter V.5 of \cite{AthreyaNey}. In our setting, the conditions of Theorem \ref{Ygen} are automatically satisfied, e.g. (G2) is essentially supported by the Perron-Frobenius Theorem for the mean semigroup matrix. 

\medskip

For the case that $E$ is countable but infinite, as in e.g.~\cite{MGWscaling, Moy},  let $M(t)$ denote the matrix whose $(i,j)$-th entry represents the mean number of type $j$ individuals alive at time $t \ge 0$ when the process is initiated from a single type $i$ individual and, for $\lambda \in \mathbb R$, set $\textstyle H(\lambda) = \int_0^\infty{\rm e}^{\lambda t}M(t)\d t$. It was shown in \cite{sandra} that if $M(t)$ is irreducible and $\lambda$ is such that $H_{i,j}(\lambda)<\infty$ for a given pair $i,j\in E$, then $H_{i,j}(\lambda)<\infty$ for all $i,j\in E$. We can then define the spectral radius $\lambda_* = \sup\{\lambda\geq  -\infty: H(\lambda) <\infty \}$. Thus, the first part of assumption (G2) holds, i.e. \eqref{earlyleftandright}, however, it is still unclear  as to whether the uniform convergence in \eqref{earlyspectralexpsgp} holds in this case. Theorems \ref{lineardecaygen} and \ref{Ygen}  also require that the spectral radius satisfies $\lambda_*=0$, but nonetheless give  new Yaglom limit results, which we do not believe is currently in the literature. 

\bigskip

We should also note that there are recent analogues of Theorem \ref{lineardecaygen} and \ref{Ygen} for superprocesses, cf. \cite{Yanxia}, however this case does not fall under our setting of branching particle systems. Moreover, the known results for superprocess are focused on local branching mechanisms. In forthcoming work we will address some of the results presented here for  superprocesses with non-local branching mechanisms. 

\medskip

The rest of the paper is set out as follows. In the next section, we will consider the evolution equations associated with the linear semigroup $(\psi_t$, $t \ge 0)$, including the many-to-one and many-to-two representations. In Section \ref{spinemoments} we outline the main steps for the proofs of Theorems \ref{genmoments}, \ref{lineardecaygen} and \ref{Ygen}, which are based around the use of the spine decomposition, together with some combinatorial arguments. Sections \ref{sec:spine} and \ref{sec:ergodic} are dedicated to the so-called spine decomposition. In particular, we provide a description this decomposition and prove a result pertaining to its ergodic behaviour, both of which will be crucial to the proof of Theorem~\ref{genmoments}. Sections \ref{sec:step1}, \ref{Sectstep2} and \ref{Sectstep3} contain the proofs of Theorems \ref{genmoments}, \ref{lineardecaygen} and \ref{Ygen}, respectively. {As alluded to earlier, it is the setting of the NTE which was the main motivation for these results in light of recent developments in this setting, cf. \cite{SNTE, SNTEII, SNTEIII, MCNTE}. Moreover, the NTE is a particular case in point where the underlying Markov process does not permit analytical formulations of the Kolmogorov and Yaglom limits which can be described in terms of PDEs. For this reason, in Section \ref{NTE}, we demonstrate   the effectiveness of our results in the  setting of the neutron branching process.}

\section{Evolution equations}
In this section, we consider some of the evolution equations associated with the linear semigroup $(\psi_t, t \ge 0)$ that will be useful in the proofs of our main results. To this end, for $f \in L_\infty^+(E)$ and $x \in E$, define
\begin{equation}
  m[f](x) := \mathcal{E}_x[\langle f, \mathcal{Z}\rangle] =  \mathcal{E}_x\left[\sum_{i = 1}^N f(x_i)\right].
\end{equation}
We will also use the notation $m(x) := m[{\bf 1}](x)$ where ${\bf 1}$ is the constant function with value $1$. We further define the
(linear) branching operator as follows:
\begin{equation}
  \mathscr{A}[f](x) \coloneqq  \gamma(x)(m[f](x) - f(x)), \qquad x\in E, f\in L_\infty^{+}(E).
  \label{linearG}
\end{equation}

Then the evolution of $(\psi_t, t \ge 0)$ can be described in the following proposition.

\begin{prop}\label{prop:linear}
  Under (G1), the semigroup $(\psi_t, t \ge 0)$ is 
  the unique solution to the  evolution equation
  \begin{equation}
    \psi_t[f](x) = \bP_t[f](x) + \int_0^t\bP_s[\mathscr{A}[\psi_{t-s}[f]]](x) \d s, \qquad t\geq0, f\in L^+_\infty(E).
    \label{lineareq}
  \end{equation}
\end{prop}

\bigskip

The proof of the above proposition follows by breaking the expectation in \eqref{linearsgp} on the first branching event and then appealing to Lemma 1.2, Chapter 4 in \cite{Dynkin2} to tidy up the resulting integral equation. The arguments are fairly standard, see for example \cite{MultiNTE}, and so we leave the details as an exercise to the reader. 

\medskip

We now consider an alternative representation of the semigroup $(\psi_t, t\geq 0)$ via a so-called many-to-one formula, which will be convenient for later.
Consider the process $Y = (Y_t,t \ge 0)$ which evolves according to the dynamics of $(\bP_t, t\geq0)$ albeit that, with instantaneous rate $\gamma(x) m(x)$, it jumps from its current position, $x\in E$, into a Borel set $A\subseteq E$ with probability  
$
m(x)^{-1}\mathcal{E}_x\left[ \mathcal{Z}(A)\right].
$ Letting ${\mathbf P} = ({\mathbf P}_x, x\in E)$ denote the law of this process, we can again follow standard reasoning highlighted above  and  condition on the first  first jump of $Y$  to deduce that
\begin{equation}
   \eta_t[f](x) \coloneqq {\mathbf{E}}_x[f(Y_t)] = \bP_t[f](x) + \int_0^t\bP_s[\gamma m \hat{\mathscr{A}}[\eta_{t-s}[f]](x)\d s,
   \label{etaevolution}
 \end{equation}
 where 
 \[
   \hat{\mathscr{A}}[f](x) = \frac{1}{m(x)}\mathcal{E}_x\left[\sum_{i = 1}^N f(x_i)\right]  - f(x).
 \]
Again  in the spirit of e.g. \cite{MultiNTE}, it is not difficult to show that one can relate the linear semigroup $(\psi_t, t\geq0)$ to the process $Y$
 via the following many-to-one formula. 
 \begin{lemma}[Many-to-one formula]\label{lem:M21}
 Suppose (G1) is satisfied. Then % for all $t \ge 0$ and all $f \in L_\infty^{+}(E)$, we have
 \begin{equation}
   \psi_t[f](x) = {\mathbf{E}}_x\left[{\rm e}^{\int_0^t B(Y_s)\d s}f(Y_t)\right],\qquad t \ge 0, f \in L_\infty^{+}(E),
   \label{M21}
 \end{equation}
 where $B(x) = \gamma(x)(m(x) - 1)$.
\end{lemma}

\medskip 

In a similar manner, we can consider a many-to-two formula. For this, recall the operator $\V$ introduced in \eqref{2point}.
%In order to do so, we define the following operator
%\begin{equation}
%  {\V}[f, g](x) = \mathcal{E}_x\Bigg[\sum_{{\substack{i, j = 1 \\ i \neq j}}}^N f(x_i)g(x_j) \Bigg],
%  \qquad f, g\in L^{+}_{\infty}(E), \, x \in E.
%\label{2point}
%\end{equation}
%We will often abuse notation and write ${\V}[g](x)$ instead of ${\V}[g,g](x)$ for $g\in L^{+}_{\infty}(E)$ and $x \in E$.

\medskip

\begin{lemma}\label{lem:M22}
Let $g\in L^{+}_{\infty}(E)$, $x \in E$ and $t \ge 0$. Suppose (G1) holds and that 
\[
\sup_{x\in E}\mathcal{E}_x[\langle 1, \mathcal{Z}\rangle^2] < \infty.
\]
Then we have the following many-to-two formula
\begin{align}
  \mathbb E_{\delta_x}\Big[\langle f, X_t\rangle \langle g, X_t\rangle \Big]
  &= \psi_{t}[fg](x)
  +\int_{0}^{t} \psi_{s}\Big[\gamma {\V}[\psi_{t-s}[f], \psi_{t-s}[g]] \Big](x)\d s.
\label{M22formula}
\end{align}
\end{lemma}

\medskip

As with the proof of Proposition \ref{prop:linear}, we have excluded the proofs of Lemma \ref{lem:M21} and Lemma \ref{lem:M22}, as they are fairly standard. In addition to \cite{MultiNTE}, we refer the reader to \cite{SNTE, MCNTE} for the proofs in the case of the NBP discussed in the introduction, and to~\cite{HR} for some general theory in the case of local branching.

\section{Three step approach to the Yaglom limit}\label{spinemoments}

The way we will approach the proofs of Theorems \ref{genmoments}, \ref{lineardecaygen} and \ref{Ygen} is through three fundamental steps.

\smallskip

\subsection{Step 1} 
The first step, which is detailed in Section \ref{sec:step1}, deals with the limiting behaviour of the martingale moments of the MBP, as presented in Theorem~\ref{genmoments}. To this end, let us consider the following martingale change of measure,
\begin{equation}
\left.\frac{\d \mathbb{P}^\varphi_\mu}{\d\mathbb{P}_\mu}\right|_{\mathcal{F}_t}  =
\frac{\langle \varphi, X_t\rangle}{\langle\varphi, \mu\rangle} \qquad t\geq 0, \mu\in\mathcal{M}(E).
\label{mgCOM}
\end{equation}
Note that this change of measure is well-defined due to the fact that under (G2), $\varphi$ is an eigenfunction, i.e.  
\begin{equation}
\label{harmonic}
\mathbb{E}_\mu [\langle\varphi , X_t\rangle ] = \langle\varphi, \mu\rangle, \qquad t\geq 0,\mu \in\mathcal{M}(E),
\end{equation} 
together with the branching Markov property. 

\medskip

Via a combination of the spine decomposition induced by \eqref{mgCOM} (which we will discuss in the next section), along with general results for Poisson processes and some combinatorial results, we will prove the following result.

\begin{theorem}\label{specialcase} 
Under the assumptions of Theorem \ref{genmoments} for some fixed $k \ge 1$, we have for all $j \le k$, 
\begin{equation}
\sup_{x\in E}\mathbb{E}_{\delta_x}^\varphi[\langle \varphi, X_t\rangle^{j}]<\infty, \qquad t\geq 0,
\label{boundedmoments2}
\end{equation}
and
\begin{equation}
\lim_{t\to\infty}\sup_{x \in E}\bigg|\frac{1}{t^{j}} \mathbb{E}^\varphi_{\delta_{x}}\left[\langle\varphi, X_t \rangle^{j} \right] -(j+1)!\left(\frac{\Sigma}{2}\right)^{j}\bigg| = 0.
\label{A2}
\end{equation}
%Equivalently, for all $k\geq 1$, we have the limiting moment behaviour under $\mathbb{P}_{\delta_{(r, \up)}}$,
%\begin{equation}
%\lim_{t\to\infty}\sup_{r \in D, \up \in V}\bigg|\frac{1}{\varphi(r,\up)t^{k}} \mathbb{E}_{\delta_{(r,\up)}}\left[\langle\varphi, X_t \rangle^{k+1} \right] -(k+1)!\left(\frac{\Sigma}{2}\right)^{k}\bigg| = 0.
%\label{eqnophimoments}
%\end{equation}
\end{theorem}
Note that, thanks to \eqref{mgCOM}, this is equivalent to Theorem \ref{genmoments}. This result is new in its own right, independently of the Yaglom limit and little seems to be known in the literature for (asymptotic) limits of spatial branching processes.

\subsection{Step 2} For the next step, which we handle in Section \ref{Sectstep2}, we prove Theorem \ref{lineardecaygen}. Our approach will take inspiration from the original approach to the Yaglom limit theorem for the neutron branching process in \S5 of \cite{MWY}. As such, we develop a non-linear integral equation for the survival probability $\mathbb{P}_{\delta_{x}}(\zeta>t) $ and show that, using an analytical approach of the operators involved, its asymptotic behaviour is linear. However, one major difference with the setting of \S5 of \cite{MWY} is that the generality of the branching mechanism in this setting forces a significantly more complex analysis of the aforementioned non-linear equation. Hence, we first obtain coarse upper and lower bounds of order $1/t$, which are then used to obtain the precise limit.  

\medskip

We note that combining steps 1 and 2 then concludes the proof of Theorem \ref{Ygen}, but only in the special case when $f = \varphi$. Indeed, 

\begin{align}
\mathbb{E}_{\delta_{x}}[\langle \varphi, X_t\rangle^k|\zeta>t] 
& =\frac{1}{t^k}\frac{\mathbb{E}_{\delta_{x}}[\langle \varphi, X_t\rangle^k\mathbf{1}_{(\zeta>t)} ]}{\mathbb{P}_{\delta_{x}}(\zeta>t)}\notag\\
& =\varphi(x)\frac{ \mathbb{E}^\varphi_{\delta_{x}}[\langle \varphi, X_t\rangle^{k-1}/ t^{k-1} ]}{ t \mathbb{P}_{\delta_{x}}(\zeta>t)}\notag \\
&\to k! \left(\frac{\Sigma}{2}\right)^k, \quad t \to \infty,
\label{prelimit}
\end{align}
which gives the desired result on account of the fact that the exponential distribution is uniquely identifiable by its moments. We refer the reader to~\cite[section 30]{billingsley} and~\cite[Theorem 8.48]{breiman} for more general results regarding classical moment convergence problems. 

\medskip

\subsection{Step 3} At this point in the argument, we have established the desired convergence of martingale moments, and hence the Yaglom limit \eqref{YaglomGen}, for the case that $f = \varphi$. In the final step, detailed in Section \ref{Sectstep3}, we remove this restriction by considering the decomposition 
\[
f = \tilde{f} + \langle \tilde\varphi, f \rangle\varphi \,\text{ where }\, \tilde{f} = f -  \langle \tilde\varphi, f \rangle\varphi.
\]
The reasoning is a relatively standard way to wrap things up, e.g. as in \cite{Ellen, Yanxia, RSS2}. We use a second moment estimate to show that, for all $x \in E$,
$
\langle\tilde f, X_t\rangle /t
$
converges weakly to zero under $\mathbb{P}_{\delta_{x}}(\cdot\, |\, \zeta>t)$, as $t\to\infty$. As such, the principal contribution in the convergence of $\langle f, X_t\rangle /t $ as $t\to\infty$ is its projection on to $\varphi$, i.e. the limit of $\langle \tilde\varphi, f \rangle\langle\varphi, X_t\rangle/t$.

\section{Spine decomposition}\label{sec:spine}
Recall that
\[
  W_t \coloneqq \frac{\langle \varphi, X_t\rangle }{\langle \mu, \varphi\rangle}, \qquad x \in E, t \ge 0,
\]
is a mean one martingale and can thus be used to define the following change of measure,
\begin{equation}
\left.\frac{\d \mathbb{P}^\varphi_\mu}{\d\mathbb{P}_\mu}\right|_{\mathcal{F}_t}  =
\frac{\langle \varphi, X_t\rangle}{\langle\varphi, \mu\rangle} \qquad t\geq 0, \mu\in\mathcal{M}(E),
\label{genmgCOM}
\end{equation}
for the $(\bP, \Phi)$-MBP. 
It is well known that for many branching processes, the change of measure \eqref{genmgCOM} induces the so-called spine decomposition. Since this will be key to the proof of Theorem \ref{genmoments}, we first spend some time discussing it.

\medskip

To this end, let us introduce another $\mathcal{M}(E)$-valued stochastic process via the following pathwise construction. 
%\begin{equation}
% X^\varphi := (X^\varphi_t, t\geq 0)\text{ with probabilities }\tilde{\mathbb{P}}^\varphi: = (\tilde{\mathbb{P}}^\varphi_{\mu}, \mu \in \mathcal{M}(E)),
% \label{dressedprocess}
% \end{equation}
% which is defined 
 \begin{itemize}
\item[1.] From the initial configuration $\mu\in \mathcal{M}(E)$ with an arbitrary enumeration of particles, the $i^*$-th individual is selected and marked `{\it spine}' so that, given $\mu$,
\begin{equation}
\text{the probability $i^* = j$ is given by }
\frac{\varphi(x_j)}{\langle\varphi, \mu\rangle}
\text{ for }j = 1,\cdots, N.
\label{howtopicki*}
\end{equation}
\item[2.] The individuals $j\neq i^*$ in the initial configuration that are not marked `{\it spine}', each issue independent copies of $(X, \mathbb{P}_{\delta_{x_j}})$ respectively.
\item [3.] The marked individual issues a single particle whose motion is determined by the semigroup
\[
  \bP_t^\varphi[f](x) \coloneqq \frac{1}{\varphi(x)}\mathbf{E}_{x}\left[{\rm e}^{\int_0^t \frac{B(Y_s)}{\varphi(Y_s)}(m [\varphi](Y_s) - \varphi(Y_s)) \d s}\varphi(Y_t)g(Y_t) \right], \qquad x \in E, f \in L_\infty^+(E).
\] 
\item[4.] The marked individual undergoes branching at the accelerated rate  
\begin{equation}
\rho(x): = \gamma(x)\frac{m[\varphi](x)}{\varphi(x)} 
\label{rho}
\end{equation} 
when  in physical configuration $x\in E$, at which point, it scatters a random number of particles according to the random measure on $E$ given by $(\mathcal{Z}, {\mathcal P}^\varphi_{x})$ where 
\begin{equation}
\frac{\d{\mathcal P}^\varphi_{x}}{\d {\mathcal P}_{x}} = \frac{\langle\varphi, \mathcal{Z}\rangle}{m[\varphi](x)}.
\label{pointprocessCOM}
\end{equation}
\item[5.] Given $\mathcal{Z}$ from the previous step, $\mu$ is redefined as $\mu = \mathcal{Z}$ and step 1 is repeated.
\end{itemize}

\medskip
 
The process above process describes the physical configuration of all the particles in the system at time $t \ge 0$, i.e. ignoring the marked genealogy. The spine decomposition induced by  \eqref{mgCOM} is summarised by the following theorem. 
%, will be denoted by $Y^\varphi: = (Y_t^\varphi)_{t\geq 0}$. Together, the processes $(X^\varphi, Y^\varphi)$ make a Markov pair, whose probabilities we will denote by $(\tilde{\mathbb{P}}^\varphi_{\mu, x}, \mu \in \mathcal{M}(E), x\in E)$. Note in particular that 
% \[
% \tilde{\mathbb{P}}^\varphi_{\mu} = \sum_{i = 1}^n\frac{\varphi (x_i)}{\langle\varphi, \mu\rangle}\tilde{\mathbb{P}}^\varphi_{\mu, x_i} 
% \]
% where $\mu = \textstyle{\sum_{i =1}^n\delta_{x_i}}$.

 \begin{theorem}
 Under assumptions (G1) and (G2), the process constructed above in 1.-5. is Markovian and equal in law to $(X, \mathbb{P}^\varphi)$, where  $ \mathbb{P}^\varphi = (\mathbb{P}^\varphi_\mu, \mu\in \mathcal{M}(E))$ was defined in \eqref{genmgCOM}. % {\color{black} Minimal conditions?}
\end{theorem}

%Let $\tilde{\mathbf{P}}^\varphi$ denote the law of the spine, $Y^\varphi$. 
We are also be interested in the evolution of the single marked genealogical line of descent, which will be referred to simply as {\it the spine}. The following lemma identifies the spine as a change of measure applied to the process $Y$ defined in Lemma \ref{lem:M21}.

 \begin{lemma}\label{spinemarkov}
Under assumptions (G1) and (G2),  the spine is equal in law to $(Y, \mathbf{P}^\varphi )$ where
  \begin{equation}
\left.\frac{\d \mathbf{P}^\varphi_{x}}{\d \mathbf{P}_{x}}\right|_{\mathcal{F}_t} =
{\rm e}^{\int_0^t B(Y_s)\d s}\frac{\varphi(Y_t)}{\varphi(x)}\mathbf{1}_{\{t<\mathtt{k}\}},\qquad t\geq 0,x \in E,
\label{NRWCOM}
\end{equation}
where $\mathtt{k} \coloneqq \inf\{t \ge 0 : Y_t \notin E\}$ and we recall that $B(x) = \gamma(x)(m(x) - 1)$. From this 
we see that the semigroup $(\emph{\texttt{P}}^\varphi_t, t\geq 0)$ associated to $(Y, \mathbf{P}^\varphi )$ is conservative and satisfies, for $f\in L^+_\infty(E)$, 
$
\emph{\texttt{P}}^\varphi_t[f] ={\varphi^{-1}}\psi_t[\varphi f]$, $t\geq 0,
$
with  stationary limit distribution 
\[
\varphi(x)\tilde\varphi(\d x), \qquad   x \in E.
\]
\end{lemma}

\bigskip

In particular, the above two results mean that under $\mathbb{E}^\varphi$, we can write
\begin{equation}
\frac{\langle\varphi, X_t\rangle}{t} = \frac{\varphi(Y_t)}{t} + \frac{1}{t}\sum_{j = 1}^{n_t} \Xi_j(Y_{T_j}, t-T_j),
\label{reordersum}
\end{equation}
where $n_t$ is the number of fission events along the spine up to time $t$, whose times are denoted $(T_j, j \geq 1)$ and, given $x\in E$ and $0\leq u<\infty$,  the random variables $\Xi_j(x,u)$ are independent and equal in law to 
\begin{equation}
\sum_{\substack{i = 1 \\ i \neq i^*}}^N
\langle\varphi, X^i_{u}\rangle 
\,\text{ under }\,\eta^\varphi_{x} \coloneqq \mathcal{P}^\varphi_{x}\bigotimes_{i = 1}^N \mathbb{P}_{\delta_{x_i}},
\label{Xi}
\end{equation}
{where $i^*$ denotes the index of the spine particle. {To be more precise, $\eta^\varphi_x$ is the law of producing offspring according to $\mathcal{P}^\varphi_x$, from which the spine index $i^*$ can be selected empirically according to \eqref{howtopicki*},  and  then  $\langle\varphi, X^i_{u}\rangle $ can be sampled for $i\neq i^*$ under the respective $\mathbb{P}_{\delta_{x_i}}$. }
%The above notation in the definition $\eta^\varphi$ means that when a branching event occurs off the spine, the number of offspring and their positions are chosen according to $\mathcal{P}^\varphi$ and then each of them issue independent copies of the original MBP. {\color{blue}Does this mean there should be an $i \neq i^*$ on the tensor too?}{\color{amethyst} I would say no because $i^*$ is an empirical selection of things produced under the tensor.}
Under $\mathbb{P}^\varphi$,  particles alive at time $t$ thus consist of the spine particle at time $t$ plus those that come from independent copies of the original process that have branching from the spine before time $t$.}
Note that, as $\Vert\varphi\Vert_\infty<\infty$, the spine term in \eqref{reordersum} converges uniformly to zero. 

\medskip

For  $x \in E$ and $t>0$, suppose we define the random variable  $\textswab{u}_t$ with law $\textswab{P}^t_{Y}$, which depends on the path of $Y$ such that
\begin{equation}
\textswab{P}^t_{Y}(\textswab{u}_t \in \d s) =\frac{ \rho(Y_s) }{\int_0^t \rho(Y_u)\d u}\d s,
\label{U_t}
\end{equation}
where $\rho$ was defined in \eqref{rho}. Recalling that the ordering of the arrivals in the sum of the Poisson arrival process $(T_i, i=1,\cdots, n_t)$ is not important, given $n_t$ and $Y$, we can appeal to standard descriptions of Poisson point processes to deduce that \eqref{reordersum} can equivalently be written 
\begin{equation}
\frac{\langle\varphi, X_t\rangle}{t} =  \frac{\varphi(Y_t)}{t}+ \frac{1}{t}\sum_{i = 1}^{n_t} \Xi_i(Y_{\textswab{u}^i_t}, t-\textswab{u}^i_t),
\label{reordersum2}
\end{equation}
where, conditional on $n_t$ and $Y$,  $\textswab{u}_t^i$ are iid  and their law is given by \eqref{U_t}.

\section{Ergodic properties of the spine}\label{sec:ergodic}
Next, we state and prove a theorem relating to the long-term behaviour of the spine, which will be of general use throughout our computations. It is loosely based on Lemma 5.1 of \cite{RSS2}. We emphasise that it only concerns the spine and not the whole branching process, thus we will only need to work with $(Y, \mathbf{P}^\varphi)$. In order to state it, let us introduce a class of functions $\mathcal{C}$ on 
 $E\times [0,\infty)\times [0,\infty)$ such that $F$ 
 belongs to class $\mathcal{C}$ if 
 \[
 F(x,s):= \lim_{t\to\infty} F(x,s, t), \qquad x \in E, s\geq 0,
 \]
  exists,
\begin{equation}
\sup_{x \in E, 0<s<1}| \varphi(x) F(x, s)| <\infty,
\label{ass1}
\end{equation}
and 
\begin{equation}
\lim_{t\to\infty} \sup_{x \in E}\varphi(x)|F(x,s)- F(x,s, t)|  = 0.
\label{ass2}
\end{equation}

\begin{theorem}\label{Ylem}
Suppose (G1) and (G2) hold. For any $k \ge 1$ and $F_1,\cdots, F_k\in\mathcal{C}$, we have
\[
\lim_{t\to\infty}\sup_{x \in E}\left| \mathbf{E}^\varphi_{x} \left[\prod_{i =1}^k \int_0^1 F_i(Y_{ut} , u, t)\d u\right] - 
\prod_{i = 1}^k\int_0^1 \langle \varphi\tilde\varphi  , F_i(\cdot, u)\rangle\d u \right|= 0,
\]
where $F_i(\cdot, u) = \lim_{t\to\infty}F_i( \cdot, u, t)$.
\end{theorem}
\begin{proof}
We will show that, for $0\leq u_1,\cdots, u_k\leq 1$ and $F_1,\cdots, F_k\in\mathcal{C}$,
\begin{equation}
\lim_{t\to\infty}\sup_{x \in E}\left| \mathbf{E}^\varphi_{x} \left[\prod_{i =1}^k  F_i(Y_{u_it} , u_i, t)\right] -\prod_{i =1}^k \langle \varphi\tilde\varphi  , F_i(\cdot, u_i)\rangle\right|=0.
\label{useinduction}
\end{equation}
Once that has been established, it is a simple matter of dominated convergence to deduce that,
\begin{align*}
& \lim_{t\to\infty}\sup_{x \in E}\left| \mathbf{E}^\varphi_{x} \left[\prod_{i =1}^k \int_0^1 F_i(Y_{ut} , u, t)\d u\right] - 
\prod_{i =1}^k\int_0^1 \langle \varphi\tilde\varphi  , F_i(\cdot, u)\rangle\d u \right|\\
&= \lim_{t\to\infty}\sup_{x \in E}\Bigg|\mathbf{E}^\varphi_{x} \left[  \int_0^1\cdots\int_0^1 
\prod_{i =1}^k
F_i(Y_{u_i t}, u_i, t)
\, \d u_1\cdots\d u_k\right]\\
&\hspace{5cm}
- 
  \int_0^1\cdots\int_0^1 \prod_{i =1}^k\langle \varphi\tilde\varphi  , F_i(\cdot, u_i)\rangle\, \d u_1\cdots\d u_k \Bigg|\\
  &\leq  \int_0^1\cdots\int_0^1 \lim_{t\to\infty}\sup_{x \in E}\Bigg|\mathbf{E}^\varphi_{x} \left[ 
\prod_{i =1}^k
F_i(Y_{u_i t}, u_i, t)
\right]
- 
  \prod_{i =1}^k\langle \varphi\tilde\varphi  , F_i(\cdot, u_i)\rangle \Bigg|\, \d u_1\cdots\d u_k\\
  &=0,
\end{align*}
as required. 
\medskip

The proof of \eqref{useinduction} works by induction. 
We start  by proving \eqref{useinduction} in the case that $k = 1$.  Suppose $F\in\mathcal{C}$.
As $\mathbf{E}^\varphi_{x} \left[ F(Y_{ut} , u, t)\right]=\texttt{P}^\varphi_{ut}[F(\cdot, u, t)](x)$, we can use Lemma \ref{spinemarkov}  and write 
\begin{align}
&\sup_{x \in E}
\left|
\texttt{P}^\varphi_{ut}[F(\cdot, u, t) ](x)- \langle \varphi\tilde\varphi  , F(\cdot, u)\rangle
\right|\notag\\
&\leq \sup_{x \in E}\frac{1}{\varphi(x)}\left|\psi_{ut}[\varphi F(\cdot, u, t)] (x)-  \psi_{ut}[\varphi F(\cdot, u)](x)\right|\notag\\
& \hspace{1cm} + \sup_{x \in E}\left|\frac{1}{\varphi(x)} \psi_{ut}[\varphi F(\cdot, u)](x) - \langle\tilde\varphi ,\varphi F(\cdot, u)\rangle\right|.
\label{twoabs}
\end{align}
Since $\psi_t$ is a linear operator, from \eqref{ass2}, for $t$ sufficiently large, the first term on the right-hand side of \eqref{twoabs} can be estimated by  $\varphi^{-1}\psi_{ut}[\varepsilon] $ for each $\varepsilon>0$. 
Hence, by the triangle inequality, the fact that  $\langle\tilde\varphi , \varepsilon\rangle \leq c\varepsilon \langle\tilde\varphi, 1\rangle $ for some constant $c > 0$ and (G2), for $t$ sufficiently large,
\begin{align}
&\sup_{x \in E}
\left|
\texttt{P}^\varphi_{ut}[F(\cdot, u, t) ](x)- \langle \varphi\tilde\varphi  , F(\cdot, u)\rangle
\right|\notag\\
&\leq \sup_{x \in E}\left|\frac{1}{\varphi(x)}\psi_{ut}[\varepsilon](x)
-\langle\tilde\varphi , \varepsilon\rangle
\right| +\varepsilon \Vert\tilde\varphi\Vert_1\notag\\
& \hspace{1cm} + \sup_{x \in E}\left|\frac{1}{\varphi(x)} \psi_{ut}[\varphi F(\cdot, u)](x) - \langle\tilde\varphi ,\varphi F(\cdot, u)\rangle\right|\notag\\
&\leq C({\rm e}^{-\alpha t}  + \varepsilon)
\label{twoabs2}
\end{align}
for some strictly positive constants $C, \alpha$, which exist thanks to (G2). This ensures that the left-hand side of  \eqref{twoabs} tends to zero as $t\to\infty$ as required.

\medskip

 Now suppose that \eqref{useinduction} is true for $k = n-1$ and,  without loss of generality, that  $0\leq u_1\leq \cdots\leq u_n\leq 1$. Then, by the Markov property,
\begin{align*}
\mathbf{E}^\varphi_{x} \left[ 
\prod_{i =1}^n
F_i(Y_{u_i t}, u_i, t)
\right] =\int_{E} \mathbf{P}^\varphi_{x}(Y_{u_1 t} \in \d y')
F_1(y', u_1, t)\mathbf{E}^\varphi_{y'} \left[ 
\prod_{i =2}^n
F_i(Y_{(u_i-u_1) t}, u_i, t)
\right].
\end{align*}
Now define 
\[
\tilde{F}(y', u_1, t) = F_1(y', u_1, t)\mathbf{E}^\varphi_{y'} 
\left[ \prod_{i =2}^n F_i(Y_{(u_i-u_1) t}, u_i, t).\right].
\]
Note, moreover, that, using the induction hypothesis,
\[
\tilde{F}(y', u_1) := \lim_{t\to\infty} \tilde{F}(y', u_1, t)= F_1(y', u_1)
\prod_{i =2}^{n} \langle \varphi\tilde\varphi  , F_i(\cdot, u_i)\rangle, \qquad y' \in E,
\]
which further respects the required boundedness conditions in the statement of the proposition.
Hence, appealing again to the induction hypothesis for $n = 1$, we get the required result. 
\end{proof}

\section{Step 1: Proof of Theorem \ref{genmoments}}\label{sec:step1}
We now prove Theorem \ref{genmoments} and thus assume that (G1), (G2) and \eqref{genmomentcondition} are in force. The strategy of the proof will be to use induction. For the inductive step, we use an interim calculation that looks at the asymptotic behaviour of contributions to the $k$-th moment that comes from the subtrees along the spine. 

\begin{proof}
Recall that Theorem \ref{genmoments} can be equivalently formulated as Theorem \ref{specialcase}.
As previously mentioned, the spine decomposition will play an important role in the proof and we will rely heavily on the decomposition \eqref{reordersum2}.  In addition to this, a number of combinatorial results for moment formulae will come into play, for which we have placed two fundamental results in the Appendix that we will call upon. The proof is long and we break it into a number of  blocks of calculations which fall under informal headings.

\medskip

{\it {\bf Block 1:} The theorem holds for  $k = 1$.}

\medskip

 We start by noting that, the moment assumption \eqref{genmomentcondition} implies that both $\textstyle\sup_{x\in E}m(x)<\infty$ and  $\textstyle\sup_{x\in E}\V[\varphi](x)<\infty$.

\medskip

Recalling $\eta^\varphi$ introduced in \eqref{Xi}, we also note that, for $x \in E$ and $0\leq u\leq t$, we have
\begin{align}
\eta^\varphi_{x}
\left[ 
\sum_{\substack{i = 1 \\ i \neq i^*}}^N
\langle\varphi, X^i_{t- u}\rangle 
 \right]\notag&=\mathcal{E}_{x}\left[\frac{\langle\varphi , \mathcal{Z}\rangle }{\mathcal{E}_{x} [\langle\varphi , \mathcal{Z}\rangle]}\sum_{j = 1}^N\frac{\varphi(x_j)}{\langle\varphi, \mathcal{Z}\rangle} \sum_{\substack{i = 1 \\ i \neq j}}^N\mathbb{E}_{\delta_{x_i}}[\langle\varphi, X^i_{t- u}\rangle]
\right]\notag\\
&=\frac{1}{\mathcal{E}_{x} [\langle\varphi , \mathcal{Z}\rangle]}
\mathcal{E}_{x}
\left[\sum_{j = 1}^N\varphi(x_j)\sum_{\substack{i = 1 \\ i \neq j}}^N\varphi(x_i)\right]\notag\\
&=\frac{{\V}[\varphi](x)}{\mathcal{E}_{x} [\langle\varphi , \mathcal{Z}\rangle]},
\label{36}
\end{align}
where we have used the martingale property from~\eqref{mgCOM} to obtain the second equality.

Using the spine decomposition \eqref{reordersum2}, we may thus write 
\begin{align}
\mathbb{E}^\varphi_{\delta_{x}}\left[\frac{\langle\varphi, X_t\rangle}{t}\right]  -
\mathbb{E}^\varphi_{\delta_{x}}\left[ \frac{\varphi(Y_t)}{t}\right] 
&= \frac{1}{t}\mathbb{E}^\varphi_{\delta_{x}}
\left[\sum_{i = 1}^{n_t}  
\frac{{\V}[\varphi]{(Y_{\textswab{u}_t^i})}}{\mathcal{E}_{Y_{\textswab{u}_t^i}} [\langle\varphi , \mathcal{Z}\rangle]}
\right]\notag\\
&= \mathbb{E}^\varphi_{\delta_{x}}
\left[
\frac{n_t}{t}
%\mathbf{E}^\varphi_{x}\left[
\frac{\int_0^t 
\rho(Y_s) \frac{{\V}[\varphi]{(Y_{s})}}{\mathcal{E}_{Y_{s}} [\langle\varphi , \mathcal{Z}\rangle]}
\d s
}
{\int_0^t  
\rho(Y_u) \d u}
%\right]
\right]\notag\\
&=
\mathbb{E}^\varphi_{\delta_{x}}\left[
\frac{n_t}{t}
%\mathbf{E}^\varphi_{x}\left[
\frac{
\int_0^t
[
\gamma(Y_s)
  {\V}[\varphi](Y_{s})/  \varphi(Y_{s} ) 
  ]\d s
}
{\int_0^t  
\rho(Y_u) \d u
}
%\right]
\right]
\notag\\
&
=
\frac{1}{t}
\mathbf{E}^\varphi_{x}\left[
\int_0^t
\frac{
\gamma(Y_s)
  {\V}[\varphi](Y_{s})
  }
  { \varphi(Y_{s} ) 
  }\d s
%\right]
\right],
\label{spine+sum3}
\end{align}
where we have used \eqref{rho} in the penultimate line and that, given $Y$, $n_t$ is Poisson distributed with mean $\textstyle \int_0^t \rho(Y_u)\d u$ to obtain the final equality. {\color{black}To show that \eqref{boundedmoments} holds, due to assumption (G1), we have $\textstyle \bar B \coloneqq \sup_{x \in E}B(x) \vee 1 < \infty$. Due to this, Fubini's theorem and the change of measure \eqref{NRWCOM}, we have
\begin{align*}
  \mathbf{E}^\varphi_{x}\left[\int_0^t\frac{\gamma(Y_s)  {\V}[\varphi](Y_{s})}{\varphi(Y_{s} )}\d s\right]
  &= \int_0^t\frac{1}{\varphi(x)}
  \mathbf{E}_x\left[ {\rm e}^{\int_0^sB(Y_u)\d u}\varphi(Y_s)\frac{\gamma(Y_s)  {\V}[\varphi](Y_{s})}{\varphi(Y_{s})}\d s\right] \\
  & \le {\rm e}^{\bar B t}
  \int_0^t \frac{1}{\varphi(x)}\mathbf{E}_x\left[\gamma(Y_s)  {\V}[\varphi](Y_{s})\right]\d s \\
  & \le {\rm e}^{\bar B t} \int_0^t  (C{\rm e}^{-\varepsilon s} + \Sigma) \d s,
\end{align*}
where we have used assumption (G2) in the final line. Due to \eqref{genmomentcondition}, the right-hand side above is finite and hence, combining this with \eqref{spine+sum3}, \eqref{boundedmoments2} (equivalently, \eqref{boundedmoments}) holds for $k = 1$.}

To prove \eqref{A2} (equivalently, \eqref{A}), returning to the right-hand side of \eqref{spine+sum3}, changing variables so that $s= ut$ and noting that the spine term will converge to $0$, we can use \eqref{spine+sum3}  to get
\begin{align}
\lim_{t\to\infty} \sup_{x \in E}\left|
\mathbb{E}^\varphi_{\delta_{x}}\left[\frac{\langle\varphi, X_t\rangle}{t}\right]
-\Sigma\right|
&=\lim_{t\to\infty} \sup_{x \in E}\left|
\mathbf{E}^\varphi_{x}\left[
\int_0^1 \frac{\gamma(Y_{ut}){\V}[\varphi](Y_{ut})}
  {\varphi(Y_{ut})} \d u\right] - \Sigma \right|
\notag\\
&=\lim_{t\to\infty} \sup_{x \in E}\left|
\int_0^1
\texttt{P}^\varphi_{ut}[
g](x)\d u
-\langle \varphi\tilde\varphi  , g\rangle\right|\notag\\
&
=0
\label{spine+sum4}
\end{align}
where  have used Theorem \ref{Ylem} with $F(x, u, t) = g(x) =\gamma(x){\V}[\varphi](x)/\varphi(x) $. 

\medskip

{\it {\bf Block 2:} Decomposing the $k$-th power of the sum in \eqref{reordersum2} assuming the theorem for $k-1$}

\medskip

We now assume that  the theorem holds for all powers less than or equal to $k - 1$ and proceed to show that the result holds true for $k$. 
We can  appeal to the spine decomposition \eqref{reordersum2} and the adaptation of the multinomial formula for expectations of iid random variables in Lemma \ref{iidmulti} in the Appendix to write
\begin{align}
&\frac{1}{t^k}\mathbb{E}^\varphi_{\delta_{x}}\left[\left(\langle\varphi, X_t \rangle - \frac{\varphi(Y_t)}{t}\right)^k \right]\notag\\
&\qquad =
\mathbb{E}^\varphi_{\delta_{x}}\left[\frac{1}{t^k}\left( \sum_{i = 1}^{n_t} \Xi_i(Y_{\textswab{u}^i_t}, t-\textswab{u}^i_t)\right)^k \right]\notag\\
&\qquad=\left(\Sigma/2\right)^{k}\mathbb{E}^\varphi_{\delta_{x}}
\Bigg[
\sum_{j =1}^{k} 2^j {n_t\choose j} \mathbf{1}_{\{j \le n_t\}}
\sum_{[k_1,\cdots,k_{j}]_+}{k \choose {k_1,\cdots, k_{j}}}\prod_{i = 1}^j (k_i-1)!
 &
\notag\\
&\hspace{5cm}
\times \prod_{i=1}^{j}
\frac{
\eta^\varphi_{Y_{\textswab{u}^i_t}}
\left[
\Xi_i(Y_{\textswab{u}^i_t}, t-\textswab{u}^i_t)^{k_i} 
\right]}
{(k_i-1)!\left(\Sigma/2\right)^{k_i-1}\Sigma t^{k_i}}
\Bigg]\notag\\
&\qquad =\left(\Sigma/2\right)^{k}\mathbb{E}^\varphi_{\delta_{x}}
\Bigg[
\sum_{j =1}^{k}  2^j 
\frac{n_t(n_t-1)\cdots(n_t-j+1) \mathbf{1}_{\{j \le n_t\}}}{j!\left(\int_0^t\rho(Y_u)\d u\right)^j}
\sum_{[k_1,\cdots,k_{j}]_+}{k \choose {k_1,\cdots, k_{j}}}\prod_{i = 1}^j (k_i-1)!
\notag\\
&\hspace{4cm}
\times \prod_{i=1}^{j} \frac{1}{t}
\int_0^t \rho(Y_{u})
\frac{
\eta^\varphi_{Y_{u}}
\left[
\Xi_i(Y_{u}, t-u)^{k_i} 
\right]}
{(k_i-1)!\left(\Sigma/2\right)^{k_i-1}\Sigma
t^{k_i-1}}\d u
\Bigg].
\label{massive}
\end{align}
{\color{black}The sums in the second and third equalities above are taken over the set $[k_1,\cdots, k_j]_+$}, which arises from the aforementioned result in the Appendix and consists of all combinations of strictly positive $\{k_1,\cdots k_j\}$ such that $\textstyle \sum_{i= 1}^j k_i= k$. %if $j\leq n_t$ and otherwise $[k_1,\cdots, k_j]_+ = \emptyset$. 

\medskip

Suppose now we write $\textstyle I_t = \int_0^t \rho(Y_{u}) \d u$ so that, given $Y$, $n_t$ has a Poisson distribution with rate $I_t$. {In particular, given $Y$, the expectation of the functional $n_t(n_t-1)\cdots (n_t-j+1)\mathbf{1}_{\{j \le n_t\}}$ is given by precisely $I_t^j$.} Hence we can develop the right-hand side of \eqref{massive} and get

\begin{align}
&\frac{1}{t^k}\mathbb{E}^\varphi_{\delta_{x}}\left[\left(\langle\varphi, X_t \rangle - \frac{\varphi(Y_t)}{t} \right)^k \right]\notag\\
%&=\left(\Sigma/2\right)^{k}
%\sum_{j =1}^{k}  2^j 
%\sum_{[k_1,\cdots,k_{j}]_+}\frac{1}{j!}{k \choose {k_1,\cdots, k_{j}}}
%\prod_{i = 1}^j (k_i-1)!
%\notag\\
%&\hspace{2cm}
%\times
%\mathbf{E}^\varphi_{x}
%\Bigg[
%\prod_{i=1}^{j} 
%\frac{\frac{1}{t}\int_0^t \rho(Y_{u}) \d u\,
%\eta^\varphi_{Y_{u}}
%\left[
%\Xi_i(Y_{u}, t-u)^{k_i} 
%\right]}{(k_i-1)!\left(\Sigma/2\right)^{k_i-1}\Sigma
%t^{k_i-1}}
%\Bigg]\notag\\
&=\left(\Sigma/2\right)^{k}
\sum_{j =1}^{k}  2^j 
\sum_{[k_1,\cdots,k_{j}]_+}\frac{1}{j!}{k \choose {k_1,\cdots, k_{j}}}
\prod_{i = 1}^j (k_i-1)!
\notag\\
&\hspace{2cm}
\times
\mathbf{E}^\varphi_{x}
\Bigg[
\prod_{i=1}^{j} 
\frac{\frac{1}{t}\int_0^t \rho(Y_{u})\,
\eta^\varphi_{Y_{u}}
\left[
\left(
\sum_{\substack{i = 1 \\ i \neq i^*}}^N
\langle\varphi, X^i_{t-u}\rangle 
 \right)^{k_i} 
\right]}{(k_i-1)!\left(\Sigma/2\right)^{k_i-1}\Sigma
t^{k_i-1}}\d u
\Bigg], \label{tohere}
\end{align}
where we recall the notation for $\eta^\varphi_x$ was given in  \eqref{Xi}. In the next two blocks, we develop the $\eta^\varphi_{Y_u}[\cdots]$ terms further and obtain asymptotics for them that allow us to complete the proof in the final block.

\medskip

{\it {\bf Block 3:} Developing the $\eta^\varphi_{Y_u}[\cdots]$ terms in \eqref{tohere} assuming the theorem for $k-1$}

\medskip

We need to control the terms 
\[
\eta^\varphi_{Y_{u}}
\left[
\left(
\sum_{\substack{i = 1 \\ i \neq i^*}}^N
\langle\varphi, X^i_{t-u}\rangle 
 \right)^{k_i} 
\right]
\]
in \eqref{tohere}, noting in particular that none of the $k_i$-s will exceed the value $k$.

\medskip

In a similar spirit to \eqref{36}, we have for any integer $1\leq p\leq k$,
   \begin{align}
&\eta^\varphi_{x}%\left( \mathcal{E}^\varphi_{x}\bigotimes_{i = 1}^{N} \mathbb{E}_{\delta_{x_i}}\right)
\left[ 
\left(\sum_{\substack{i = 1 \\ i \neq i^*}}^N
\langle\varphi, X^i_{t- u}\rangle
\right)^p
 \right]\notag\notag\\
 &=\mathcal{E}_{x}\left[\frac{\langle\varphi , \mathcal{Z}\rangle }{\mathcal{E}_{x} [\langle\varphi , \mathcal{Z}\rangle]}\sum_{j = 1}^N\frac{\varphi(x_j)}{\langle\varphi, \mathcal{Z}\rangle} 
 \bigotimes_{i = 1}^{N} \mathbb{E}_{\delta_{x_i}}
 \left[
 \left(
 \sum_{\substack{i = 1 \\ i \neq j}}^N\
 \langle\varphi, X^i_{t- u}\rangle
\right)^p
\right]
\right]\notag\\
&=\frac{1}{\mathcal{E}_{x} [\langle\varphi , \mathcal{Z}\rangle]}
\mathcal{E}_{x}
\left[\sum_{j = 1}^N\varphi(x_j)
%\bigotimes_{i = 1}^{N} \mathbb{E}_{\delta_{x_i}}
\sum_{
\substack{p_j = 0 \\ [p_1,\cdots,p_{N}]} 
}
{p \choose p_1,\cdots, p_{N}}
\prod_{i = 1}^N 
 \mathbb{E}_{\delta_{x_i}}[\langle\varphi, X_{t-u}\rangle^{p_i}]
\right],
\label{36again}
\end{align}
   where in the final equality we have used the multinomial formula in Lemma \ref{ais} of the Appendix; in particular, the sum is over the set $[p_1,\cdots, p_N]$ of all combinations of $(p_i, i= 1,\cdots N)$ such that $\textstyle \sum_{i= 1}^N p_i= p$.

\medskip

We start by noting that there are no more than $p$ strictly positive elements in the product on the right-hand side above since, regardless of the value of $N$, the set $\{1, \dots, p\}$ cannot be partitioned into more than $p$ non-empty blocks. 
 {\color{black} Then, the induction hypothesis (set as valid for $k-1$), the fact that $\textstyle  \sum_{[p_1, \cdots, p_N] } {p \choose p_1,\cdots, p_{N}} = N^p$, the boundedness of $\varphi$ and the assumption \eqref{genmomentcondition} ensures that, uniformly in $x$, the entirety of the right-hand side, and hence the left-hand side, of \eqref{36again} is finite.
 
To show that this implies \eqref{boundedmoments}, we first note that using the same techniques as in the proof of Theorem \ref{Ylem}, one can show that, under the assumptions of this theorem, there exists a constant $K > 0$ such that
\[
\sup_{x \in E}\left| \mathbf{E}^\varphi_{x} \left[\prod_{i =1}^j \int_0^1 F_i(Y_{ut} , u, t)\d u\right] - 
\prod_{i = 1}^j\int_0^1 \langle \varphi\tilde\varphi  , F_i(\cdot, u)\rangle\d u \right| \le K,
\]

Then, feeding \eqref{36again} into the right-hand side of \eqref{tohere}, given the definition of $\rho$ in \eqref{rho}, and the above inequality, we see that \eqref{boundedmoments} holds. 
 }
 \medskip
 
 {\it {\bf Block 4:} Uniform asymptotics of $\eta^\varphi_{Y_u}[\cdots]$ terms in \eqref{tohere}  assuming the theorem for $k-1$}
  \medskip

 It follows from \eqref{36again} that, for $t\ge 1$,
   \begin{align*}
  &\sup_{u\leq t}\Bigg| \frac{\mathcal{E}_{x} [\langle\varphi , \mathcal{Z}\rangle]}{
t^{k-1}}
%\left( \mathcal{E}^\varphi_{x}\bigotimes_{i = 1}^{N} \mathbb{E}_{\delta_{x_i}}\right)
\eta^\varphi_{x}
\Bigg[ 
\Bigg(\sum_{\substack{i = 1 \\ i \neq i^*}}^N
\langle\varphi, X^i_{t- u}\rangle
\Bigg)^k 
\Bigg]\\
&
\hspace{4cm}-  \frac{1}{
t^{k-1}}
\mathcal{E}_{x}
\Bigg[\sum_{j = 1}^N\varphi(x_j)
\sum_{
\substack{i=1\\  i\neq j} 
}^N
 \mathbb{E}_{\delta_{x_i}}[\langle\varphi, X_{t-u}\rangle^{k}]
\Bigg]\Bigg|\\
&\leq 
\frac{1}{t}
\mathcal{E}_{x}
\left[\sum_{j = 1}^N\varphi(x_j)
%\bigotimes_{i = 1}^{N} \mathbb{E}_{\delta_{x_i}}
\sum_{
\substack{k_j = 0 \\ [k_1,\cdots,k_{N}]_2} 
}
{k \choose k_1,\cdots k_{N}}
\prod_{i = 1}^N 
 \sup_{u \le t}\mathbb{E}_{\delta_{x_i}}[\langle\varphi, X_{s}\rangle^{k_i}]
 /t^{k_i -1}
\right]
\\
&\leq 
\frac{1}{t}
\mathcal{E}_{x}
\left[\sum_{j = 1}^N\varphi(x_j)
%\bigotimes_{i = 1}^{N} \mathbb{E}_{\delta_{x_i}}
\sum_{
\substack{k_j = 0 \\ [k_1,\cdots,k_{N}]_2} 
}
{k \choose k_1,\cdots k_{N}}
\prod_{i = 1}^N 
\left(
\frac{k_i}{k_i-1}\right)^{k_i} \varphi(x_i)
\frac{\mathbb{E}^\varphi_{\delta_{x_i}}[\langle\varphi, X_{t}\rangle^{k_i-1}]
 }{t^{k_i -1}}
\right]
   \end{align*}
   where $[k_1,\cdots,k_N]_2$ are the combinations of non-negative integers $k_1,\cdots,k_N$ such that $\textstyle \sum_{i = 1}^N k_i = k$, with at least two of the $k_i >0$. {\color{black} Using similar arguments to the treatment of the right-hand side of \eqref{36again},   the  assumption \eqref{genmomentcondition}}, the fact that $\Vert\varphi\Vert_\infty<\infty$ and the induction hypothesis at $k-1$ gives us that, uniformly in $x \in E$, the right-hand side tends to zero. That is to say 
   \begin{align}
&\lim_{t\to\infty}\sup_{x \in E}   \sup_{u\leq t}\Bigg| \frac{\mathcal{E}_{x} [\langle\varphi , \mathcal{Z}\rangle]}{
t^{k-1}}
\eta^\varphi_{x}
\Bigg[ 
\Bigg(\sum_{\substack{i = 1 \\ i \neq i^*}}^N
\langle\varphi, X^i_{t- u}\rangle
\Bigg)^k 
\Bigg]\notag\\
&
\hspace{6cm}-  \frac{1}{
t^{k-1}}
\mathcal{E}_{x}
\Bigg[\sum_{j = 1}^N\varphi(x_j)
\sum_{
\substack{i=1\\  i\neq j} 
}^N
 \mathbb{E}_{\delta_{x_i}}[\langle\varphi, X_{t-u}\rangle^{k}]
\Bigg]\Bigg|=0.
\label{mega}
\end{align}

This gives the uniform leading order behaviour of the  term on the left-hand side of \eqref{36again}.

\medskip

{\it {\bf Block 5:} Uniform asymptotic martingale moments  assuming the theorem for $k-1$}

\medskip

If we can show that, for each $j\leq k$,
\begin{align}
\label{B}
&\lim_{t\to\infty}\sup_{x\in E}
\left|
\mathbf{E}^\varphi_{x}
\left[
\prod_{\ell=1}^{j} 
\frac{\frac{1}{t}
\int_0^t  \d u\,\rho(Y_{u})
\eta^\varphi_{Y_u}
\big[
\big(\sum_{\substack{i = 1 \\ i \neq i^*}}^N
\langle\varphi, X^i_{t- u}\rangle \big)^{k_\ell}
\big]}{(k_\ell-1)!\left(\Sigma/2\right)^{k_\ell-1}\Sigma
t^{k_\ell-1}}
\right] - 1
\right| = 0,
\end{align}
then this, Lemma \ref{combinatoriclemma} in the Appendix and \eqref{tohere} yield
\begin{align*}
\lim_{t\to\infty}
\frac{1}{t^k}\mathbb{E}^\varphi_{\delta_{x}}\left[\left(\langle\varphi, X_t \rangle - \frac{\varphi(Y_t)}{t}\right)^k \right]
&=\frac{1}{t^k}\mathbb{E}^\varphi_{\delta_{x}}\left[\langle\varphi, X_t \rangle^k \right]\\
&=\left(\Sigma/2\right)^{k}
\sum_{j =1}^{k}  2^j 
\sum_{[k_1,\cdots,k_{j}]_+}\frac{1}{j!}{k \choose {k_1,\cdots, k_{j}}}
\prod_{i = 1}^j (k_i-1)!\\
&=\left(\Sigma/2\right)^{k}(k+1)!
\end{align*}
uniformly for $x \in E$, which completes the proof of the theorem. Note that we have used the fact that $\varphi(Y_t)/t \to 0$ as $t \to \infty$ in the first equality above.

\medskip

The remainder of the proof is thus focused on showing \eqref{B}. Suppose that $1\leq j\leq k$ and $k_1,\cdots, k_j\in \mathbb{N}$ are such that $\textstyle \sum_{i =1}^j k_i = k$. Then, from \eqref{mega} used in conjunction with \eqref{tohere} we have 
\begin{align}
&\lim_{t\to\infty}\sup_{x\in E}\left|
\mathbf{E}^\varphi_{x}
\left[
\prod_{\ell=1}^{j} \frac{\frac{1}{t}\int_0^t \d u\,\rho(Y_{u})
\eta^\varphi_{x}
\big[
\big(\sum_{\substack{i = 1 \\ i \neq i^*}}^N
\langle\varphi, X^i_{t- u}\rangle \big)^{k_\ell}
\big]}{(k_\ell - 1)! \left(\Sigma/2\right)^{k_\ell-1}\Sigma
t^{k_\ell-1}}
\right] - 1
\right|
\notag\\
&
 =\lim_{t\to\infty}\sup_{x\in E}  \left|
\mathbf{E}^\varphi_{x}
\left[
\prod_{\ell=1}^{j} \frac{
\frac{1}{t}
\int_0^t \d u\,
\frac{\gamma(Y_{u})}{\varphi(Y_{u})} 
\,
\mathcal{E}_{Y_u}
\Bigg[\sum_{j = 1}^N\varphi(Y_j)
\sum_{
\substack{i=1\\  i\neq j} 
}^N
 \mathbb{E}_{\delta_{Y_i}}[\langle\varphi, X_{t-u}\rangle^{k_\ell}]
\Bigg]
}{(k_\ell - 1)! \left(\Sigma/2\right)^{k_\ell-1}\Sigma
t^{k_\ell-1}}
\right] - 1
\right|\notag\\
&
  = \lim_{t\to\infty}\sup_{x\in E} \frac{1}{\Sigma^j}\left|
\mathbf{E}^\varphi_{x}
\left[
\prod_{\ell=1}^{j} 
\int_0^1 \d s\, F_\ell(Y_{st},s,t)
\right] - \Sigma^j
\right|,
\label{uniformtozero}
   \end{align}
where, for $\ell = 1,\cdots, j$, $x \in E$, $0\leq u\leq 1$ and $t\geq 0$,

 \begin{align*}
   F_\ell(x,s,t) &=\frac{\gamma(x)}{\varphi(x)} 
\,
\mathcal{E}_{x}
\Bigg[\sum_{j = 1}^N\varphi(x_j)
\sum_{
\substack{i=1\\  i\neq j} 
}^N\varphi(x_i)
 \frac{ \mathbb{E}^\varphi_{\delta_{x_i}}[\langle\varphi, X_{t(1-s)}\rangle^{k_\ell-1}]
}{(k_\ell - 1)! \left( \Sigma/2\right)^{k_\ell-1}
t^{k_\ell-1}}
\Bigg]\\
&= k_\ell(1-s)^{k_\ell - 1}\frac{\gamma(x)}{\varphi(x)} 
\,
\mathcal{E}_{x}
\Bigg[\sum_{j = 1}^N\varphi(x_j)
\sum_{
\substack{i=1\\  i\neq j} 
}^N\varphi(x_i)
 \frac{ \mathbb{E}^\varphi_{\delta_{x_i}}[\langle\varphi, X_{t(1-s)}\rangle^{k_\ell-1}]
}{k_\ell! \left( \Sigma/2\right)^{k_\ell-1}
(t(1-s))^{k_\ell-1}}
\Bigg].
\end{align*}
   Since we are assuming that the statement of the theorem  holds for $k - 1$, it is easy to verify 
 with the help of \eqref{genmomentcondition} that $F_i\in \mathcal{C}$ (where $\mathcal{C}$ was defined immediately before Theorem \ref{Ylem}) and that 
   \[
F_\ell(x,s): = \lim_{t\to\infty}F_\ell(x,s,t) = k_\ell (1-s)^{k_\ell - 1}\frac{\gamma(x)}{\varphi(x)} {\V}[\varphi](x).
   \]
   Note also that $k_\ell(1-s)^{k_\ell - 1}\Sigma = \langle\varphi\tilde\varphi , F_\ell(\cdot, s)\rangle$ so that $\textstyle \int_0^1 \langle\varphi\tilde\varphi , F_\ell(\cdot, s)\rangle \d s = \Sigma$.
   As a consequence, Theorem \ref{Ylem} tells us that the right-hand side of \eqref{uniformtozero} converges to zero, as required.
\end{proof}

\section{Step 2: Proof of Theorem \ref{lineardecaygen}}\label{Sectstep2}

In addition to (G1) and (G2) assumed in the previous section, we further assume (G3), (G4) and (G5). Note that (G5) implies that the previously enforced moment condition \eqref{genmomentcondition} holds.

\smallskip

For bounded functions $g: E\to [0, 1]$, let us define the non-linear semigroup
\[
u_{t}[g](x)=\mathbb E_{\delta_{x}}\left[1-\prod_{i=1}^{N_{t}}g(x_{i} (t))\right], \quad x \in E, 
\]
with the convention that if $\zeta\leq t$ then the above product is equal to $1$. We will also use the notation
\begin{equation}
u_{t}(x):=u_{t}[\mathbf 0](x)=\mathbb P_{\delta_{x}}(\zeta>t) \ \text{ and }\  a(t):=\langle\tilde\varphi, u_{t}\rangle,
\label{u0at}
\end{equation}
where $\mathbf 0$ is the zero function.

\begin{prop}
\label{prop: diff-eq}
For all bounded and measurable functions $g\in L^+_\infty(E)$ such that $\Vert g\Vert_\infty\leq 1$ and $x \in E$, $t\geq 0$, $u_{t}[g]$ satisfies
\begin{equation}
\label{eq: diff-eq}
u_{t}[g](x)=\psi_{t}[1-g](x)+\int_{0}^{t}\psi_{s}\big[\G[u_{t-s}[g]]\big](x)\d s,
\end{equation}
where the operator $\G$ is given by 
\begin{equation}
\label{def: scrG}
\G[h](x):=\gamma(x)\mathcal E_{x}\left[1-\prod_{i=1}^{N}\big(1-h(x_{i})\big)-
\sum_{i=1}^N h(x_i) \right] 
\end{equation}
for $h\in L^+_{\infty, 1}(E)$.
It follows that $a(t)$,  defined in \eqref{u0at}, satisfies
\begin{equation}
\label{eq: diff-eq'}
a(t)=a(0)+\int_{0}^{t} \langle \tilde\varphi, \G[u_{t-s}]\rangle \d s=a(0)+\int_{0}^{t} \langle \tilde\varphi, \G[u_{s}]\rangle \d s,\qquad t\geq 0. 
\end{equation}
\end{prop}

\begin{proof}
{\color{black} Again, the proof uses the usual method of splitting on the first branching event, see for example \cite{MultiNTE, SNTE, SNTEII}, so we will only sketch the arguments here. Recall that $(\bP_t, t\geq0)$ is the underlying motion and %Let us first recall the advection semigroup:
%\[
%\bP_{t}[g](x)=\bE_x[g(\xi_t)\mathbf 1_{t<\zeta}], \ \text{ with } \ \zeta:=\inf\{t>0: \xi_t \notin E\}.
%\]
%Recall also 
that
the branching operator $\mathscr{A}$ given in \eqref{linearG}. 
By conditioning on the first branching event, and then applying Lemma 1.2, Chapter 4 in \cite{Dynkin1}, we find that, for $x \in E, t\geq0$ and $g\in L^+_{\infty, 1}(E)$,
\begin{equation}\label{eq: interm}
u_{t}[g](x)=\bP_{t}[1-g](x) +\int_{0}^{t} \bP_{s}\big[\mathscr{A}[u_{t-s}[g]]\big](x)\d s+\int_{0}^{t} \bP_{s}\big[\G[u_{t-s}[g]\big](x)\d s. 
\end{equation}
Next, recall the process $(Y_{t}, t\ge 0)$ under $\mathbf P_{x}$ defined just before Lemma \ref{lem:M21}. 
%Writing $\tau^{D}=\inf\{t>0: Y_{t}\notin D\}$ for the exit time of this process from $E$, 
Define $(\tilde u_t, t\geq 0)$ via 
\begin{align}\label{eq: interm'}
\tilde u_{t}[g](x)=\,\mathbf E_{x}\Big[{\rm e}^{\int_{0}^{t}B(Y_{s})\d s}\big(1-g(Y_{t})\big)
%\mathbf 1_{(t<\tau^{D})}
\Big] +\mathbf E_{x}\Big[\int_{0}^{t }{\rm e}^{\int_{0}^{s}B(Y_{u})\d u}\G\big[\tilde u_{t-s}[g]\big](Y_{s})\Big],
\end{align}
for $x \in E, t\geq 0$ and $g\in L^{+}_{\infty, 1}(E)$.
Note that for the moment we don't claim a solution to \eqref{eq: interm'} exists. By conditioning the right-hand side of \eqref{eq: interm'} on the first jump of $Y$ and invoking Lemma 1.2, Chapter 4 of \cite{Dynkin1} (bearing in mind the dynamics of $Y$ given just before Lemma \ref{lem:M21})
we can check that $(\tilde u_{t}, t\geq 0)$ is a solution to  \eqref{eq: interm'} if and only if it solves \eqref{eq: interm}. In particular,  $(u_{t}, t\geq 0)$ satisfies \eqref{eq: interm'}. Combined with the many-to-one formula, this implies that
\begin{equation}
u_{t}[g](x)=\psi_{t}[1-g](x)+\int_{0}^{t}\psi_{s}\big[\G[u_{t-s}[g]]\big](x)\d s, 
\label{useg0}
\end{equation}
which is \eqref{eq: diff-eq}. As $\tilde\varphi$ is a left eigenfunction for the semigroup $(\psi_t, t\geq0)$ and as we are dealing with the critical case,  we have, for all $t\geq0$, that  $\langle\tilde\varphi, \psi_t[g]\rangle = \langle \tilde\varphi,g \rangle$. Using this fact and  taking $g=\mathbf 0$ in \eqref{useg0}, we find that
\[
a(t)=\langle \tilde\varphi, \psi_{t}[\mathbf 1]\rangle+\int_{0}^{t} \langle \tilde \varphi, \psi_{s}\big[\G[u_{t-s}]\big]\rangle\d s
=\langle \tilde \varphi, \mathbf 1\rangle + \int_{0}^{t} \langle \tilde \varphi, \G[u_{t-s}]\rangle\d s, 
\]
where $\mathbf{1}$ is the function which is identically unity. The above equality gives rise to the first identity in \eqref{eq: diff-eq'}. The second identity in \eqref{eq: diff-eq'} then follows by a change of variable. }
\end{proof}

Our aim is to extract the proof of Theorem \ref{lineardecaygen} from \eqref{eq: diff-eq} and \eqref{eq: diff-eq'}. To this end, we need a series of technical lemmas. The first lemma regards the analytical behaviour of the ${\V}[h](x)$ defined in \eqref{2point} and $\G[h](x)$ defined in \eqref{def: scrG}.

\begin{lemma}
\label{lem: prop-G}
The following statements hold. 
\begin{enumerate}[(i)]
\item
For all $x \in E$ and $h: E\to [0, 1]$, we have $0\ge \G[h](x)\ge -\norm{\gamma}_{\infty}M$, where $M  = \sup_{x \in E}m(x)$. 
\item
There exists $C\in (0, \infty)$ such that, for all $h: E\to [0, \tfrac12]$, 
\begin{equation}
\label{eq: bd-gg}
\sup_{x \in E}\Big|\G[h](x)+\tfrac12\gamma(x){\V}[h](x)\Big|\le C \sup_{x \in E}h(x)^{3}.
\end{equation}
\item
{\color{black}There exists $C\in (0, \infty)$ such that for all $h: E\to [0, \tfrac12]$, 
\[
-\G[h](x)\ge C\gamma(x) {\V}[h](x) \ge 0, 
\qquad x \in E.
\]}
\item
There exists $C\in (0, \infty)$ such that for all $h_{1}, h_{2}: E\to [0, 1]$, we have
\[
\sup_{x \in E}\Big|{\V}[h_{1}](x) -{\V}[h_{2}](x)\Big| \le C\sup_{x \in E}|h_{1}(x)-h_{2}(x)|. 
\]
\end{enumerate}
\end{lemma}

\begin{proof} (i)
The non positivity of $\G[h]$ can be demonstrated using an inductive argument. For $n\ge 1$ and $(x_{i}, 1\le i\le n) \in [0, 1]^{n}$,
\[
f_{n}(x_{1}, \dots, x_{n}):=1-\prod_{i=1}^{n}(1-x_{i})-\sum_{i=1}^{n}x_{i}\le 0.  
\]
To see why this is true, we note that $f_{1}(x_{1})=0$ and more generally, we have $f_{n+1}(x_{1}, \dots, x_{n}, 0)=f_{n}(x_{1}, \dots, x_{n})$ as well as $\partial_{n+1}f_{n+1}(x_{1}, \dots, x_{n+1})\le 0$. This yields $\G[h]\leq 0$.

\medskip

For the second inequality in (i), it suffices to observe that since $0\le h\le 1$
\[
 \mathcal{E}_x\left[1-\prod_{i}(1-h(x_{i}))-\langle h, \mathcal Z\rangle\right]\ge -\mathcal{E}_x\left[\langle h, \mathcal Z\rangle \right]\ge -\sup_{x \in E}m(x),
\]
which is bounded due to (G1).

\medskip

(ii) Let us write $L(h)=-\langle \log(1-h), \mathcal Z\rangle$, which is non negative. Then, we have
\begin{align}
\Big|\G[h](x)+\tfrac12\gamma(x){\V}[h](x)\Big|&\le \gamma(x)\mathcal E_{x}\big[1-{\rm e}^{-L(h)}-L(h)+\tfrac12 L(h)^{2}\big]\notag\\
&\hspace{0.5cm}+\gamma(x)\mathcal E_{x}\big[\langle -\log(1-h)-h-\tfrac12 h^{2}, \mathcal Z\rangle\big]\notag\\
&\hspace{1cm}+ \tfrac12\gamma(x)\mathcal E_{x}\big[\big|\langle \log(1-h), \mathcal Z\rangle^{2}-\langle h, \mathcal Z\rangle ^{2}\big|\big].
\label{3terms}
\end{align}
Applying the elementary bounds that $0\le 1-{\rm e}^{-x}-x+\tfrac12x^{2}\le \tfrac16 x^{3}$ for all $x\ge 0$ and that $0\le -\log(1-x)-x-\tfrac12x^{2}\le x^{3}$ for $x\in [0, \tfrac12]$, since both $\gamma$ and the number of offspring produced at a branching event are uniformly bounded thanks to (G1) and (G4), respectively, we see the first two terms on the right-hand side of \eqref{3terms} are bounded by $C\norm{h}_{\infty}^{3}$. For the third term on the right-hand side of \eqref{3terms}, writing $\log(1-h)=-h+R(h)$ with $|R(h)|\le C\|h\|^{2}_{\infty}$, we get
\[
\big|\langle \log(1-h), \mathcal Z\rangle^{2}-\langle h, \mathcal Z\rangle ^{2}\big|
\le 2\langle h, \mathcal Z\rangle\cdot\langle |R(h)|, \mathcal Z\rangle +\langle R(h), \mathcal Z\rangle^{2}
\le C \,\norm{h}_{\infty}^{3}, 
\]
which leads to (ii). 
\medskip

(iii) First note that for all $n\ge 2$ and $(x_{i},1\le i\le n)\in[0, \tfrac12]^{n}$, 
\[
\prod_{i=1}^{n}(1-x_{i})-1+\sum_{i=1}^{n}x_{i}\ \ge \ \frac{1}{2^{n-1}}\sum_{1\le i, j\le n: i\ne j}x_{i}x_{j}. 
\]
This can be shown with an induction on $n$ and its proof is therefore omitted. The result then follows due to (G4).

\medskip

(iv) The final claim can be checked from the definition of ${\V}$, combined with (G4).
\end{proof}

\medskip

We now prove a sequence of results that show that $a(t)$ and $u_t(x)$ can be bounded above and below by a constant divided by $t$. We can then use this to show that the difference between $u_t(x)/\varphi$ and $a(t)$ decays sufficiently quickly as $t \to \infty$, which we then, in turn, use to obtain the precise constant given in Theorem \ref{lineardecaygen}. The proof is inspired by \S5 of \cite{MWY}, where Theorem~\ref{lineardecaygen} was proved for the NBP in the isotropic case with homogeneous branching. In this case, equation \eqref{u0at} and the operator $\V$ simplify due to the lack of correlation between particles produced at branching events. This is not possible in the general case, as particles may be correlated. 

\medskip

We start with the lower bound.

\begin{lemma}\label{lem: a-lowerbd}
There exists $C \in (0, \infty)$ such that 
\[
\mathbb P_{\delta_{x}}(\zeta>t) \ge \frac{\varphi(x)}{Ct} \qquad \text{ and } \qquad a(t)\ge \frac{1}{Ct}
\]
for all $t\ge 1$. 
\end{lemma}

\begin{proof}
Let us recall the change of measure \eqref{genmgCOM} and that, from Theorem \ref{genmoments},  there exists a $C>0$ such that 
\[
\sup_{x\in E}\mathbb E^{\varphi}_{\delta_{x}}[\langle \varphi, X_{t}\rangle]\le Ct
\]
 for all $t\ge 1$. By Jensen's inequality, we then get
\begin{equation}
\mathbb P_{\delta_{x}}(\zeta>t)=\mathbb E^{\varphi}_{\delta_{x}}\left[\frac{\varphi(x)}{\langle \varphi, X_{t}\rangle}\right]\ge \frac{\varphi(x)}{\mathbb E^{\varphi}_{\delta_{x}}[\langle\varphi, X_{t}\rangle]}\ge \frac{\varphi(x)}{Ct}, \quad t\ge 1. 
\label{typesofh}
\end{equation}
The lower bound for $a(t)$ then follows from an integration with $\tilde\varphi$, recalling that we have normalised the left and right eigenfunctions so that $\langle\tilde{\varphi}, \varphi\rangle =1$.
\end{proof}

Before turning to the upper bound, we first state and prove a lemma regarding the extinction time of the process.

\begin{lemma}\label{lem: unif-small} 
Under the assumptions of Theorem \ref{lineardecaygen}, there exists a constant $C >0 $ such that for all $t\ge t_{0}$, 
\begin{equation}
   a(t) \le \frac{C}{t}\,\text{ and } \,\sup_{x \in E}u_{t}(x) \le \frac{C}{t}.
   \label{eq: a-upperbd}
\end{equation}
\end{lemma}

\begin{proof}
We first show that 
\[
  a(t) \to 0 \, \text{ and }  \,\sup_{x \in E}u_{t}(x) \to 0
\]
as $t \to \infty$.

\medskip

Recall $\zeta=\inf\{t>0: \langle1, X_t\rangle = 0\}$ is the extinction time. Due to (G3), we have $\mathbb P_{\delta_{x}}(\zeta<\infty)=1$, namely, $u_{t}(x)=\mathbb P_{\delta_{x}}(\zeta>t)\to 0$, as $t \to \infty$, for all $x \in E$. 
%Clearly, for each $x$, $t\mapsto u_{t}(x)=\mathbb P_{\delta_{x}}(\zeta>t)$ is non increasing. 
%Let us write $u_{\infty}(x)=\lim_{t\to\infty}u_{t}(x)$. Similarly, we denote by $a(\infty)$ the limit of $a(t)$.  Letting $t\to\infty$ in \eqref{eq: diff-eq'}, we find by monotone convergence that, 
%\[
%0\le -\int_{0}^{\infty}\langle \tilde\varphi, \G[u_{s}]\rangle\d s=a(0)-a(\infty)<\infty. 
%\]
%In particular, this implies that $\textstyle \limsup_{t\to\infty}\langle \tilde\varphi, -\G[u_{t}]\rangle=0$. On the other hand, $u_{t}(x)\to u_{\infty}(x)$ for each $x$. Due to Lemma \ref{lem: prop-G} (i), it is straightforward to check that $\G[u_{t}](x)\to \G[u_{\infty}](x)$ for each $x$, and that $\textstyle0=\limsup_{t\to\infty}\langle \tilde\varphi, -\G[u_{t}]\rangle=\langle \tilde\varphi, -\G[u_{\infty}]\rangle=0$. Hence, $\G[u_{\infty}](x)=0$ for a.e.~$x$. According to Lemma \ref{lem: prop-G} (i), this implies $u_{\infty}(x)=0$ for a.e.~$x$.
Therefore, $a(t)=\langle \tilde\varphi, u_{t}\rangle\to 0$, as $t\to\infty$, by dominated convergence. To get to the uniform convergence of $u_{t}$ to zero, let us note that $u_{t+s}(x)=u_{t}[1-u_{s}](x)$ by the Markov branching property. Appealing to \eqref{eq: diff-eq} and Lemma \ref{lem: prop-G} (i), we find that
\begin{equation}\label{eq: ut-bd}
0\le u_{t+s}(x)=\psi_{t}[u_{s}](x)+\int_{0}^{t}\psi_{l}\big[\G[u_{t+s-l}]\big](x)\d l\le \psi_{t}[u_{s}](x),
\end{equation}
by Lemma \ref{lem: prop-G} (i). Combined with (G2), this yields
\begin{equation}
\norm{u_{t+s}}_{\infty}\le \norm{\psi_{t}[u_{s}]}_{\infty}\le a(s)\norm{\varphi}_{\infty}+O({\rm e}^{-\varepsilon t}). 
\label{atou}
\end{equation}
Taking first $t\to\infty$ and then $s\to\infty$ gives the desired limit $||u_t||\to0$ as $t\to0$. 

\medskip

Now we prove the required upper bound on $a(t)$ and $\norm{u_{t}}_{\infty}$. Due to the above calculations, we may fix $t_0 > 0$ such that $\norm{u_{t}}_{\infty} \le 1/2$ for all $t\ge t_{0}$. 
Note that the integrand in \eqref{eq: diff-eq'} is bounded due to our assumptions and Lemma \ref{lem: prop-G} (i). It follows that $a(t)$ is differentiable. 
Differentiating   \eqref{eq: diff-eq'} for $t\ge t_{0}$ and then applying the bound in Lemma \ref{lem: prop-G} (iii), we obtain that for $t\ge t_{0}$, 
\begin{align}
\notag
a'(t) =  \langle \tilde \varphi, \mathscr{G}[u_{t}]\rangle 
\le -C_{1}\left\langle \tilde \varphi,  \gamma{\V}[u_t]\right\rangle 
\le - C_{2} \langle \tilde\varphi, u_t\rangle^2
= -C_2a(t)^2,   \label{bd: at1}
\end{align}
where we have used assumption (G5) in the second inequality. 

%Note that  H\"{o}lder's inequality implies that 
%\[
%\left(\int_{D} f(r) \d r\right)^{3} \le C_{3}\int_{D} |f(r)|^{3} \d r. 
%\] 
%Applying this to $r\mapsto \int_{V}\tilde\varphi(x) u_t(x)\d\up$, we find that
%\begin{align*}
%  a(t)^3 &= \left(\int_{D \times V}\tilde\varphi(x) u_t(x)\d r\d\up \right)^3  \le C_3\int_D\left(\int_V\tilde\varphi(x)u_t(x)\d\up \right)^3\d r \\
%  &= C_3\int_D \d r \int_{V \times V \times V}\tilde\varphi(x)\tilde\varphi(r, \up_1)\tilde\varphi(r, \up_2)
%  u_t(x)u_t(r, \up_1)u_t(r, \up_2)\d\up\,\d\up_1\d\up_2 \\
%  &\le C_4\norm{u_{t}}_{\infty} \int_{D \times V}\tilde\varphi(x)\int_{V}\int_V u_t(r, \up_1)u_t(r, \up_2)\d r\,\d\up\,\d\up_1\d\up_2,
%\end{align*}
%since $\tilde\varphi$ is uniformly bounded. Comparing this with \eqref{bd: at1}, we have
%\[
%a'(t)  \le - \frac{C_{5}}{\norm{u_{t}}_{\infty}} a(t)^{3}, \quad t\ge t_{0}. 
%\]
%We have seen that $\norm{u_{t}}_{\infty}\to 0$. Therefore, for any fixed $\epsilon>0$, there exists $t'_{0}=t'_{0}(\epsilon)$ so that 
%\[
%a'(t) \le  - \frac{C_{5}}{\epsilon^{2}} a(t)^{3}, \quad t\ge t_{1}:=\max(t_{0}, t'_{0}). 
%\]

\medskip

Integrating from $t_{0}$ to $t$ yields 
\[
a(t)\le \left(C_{2}(t-t_{0})+a(t_{0})^{-1}\right)^{-1} \le \left(C_{2}t\right)^{-1}, 
\]
where the last inequality holds for $t$ sufficiently large.
The upper bound for $a(t)$ then follows. We may then apply the same techniques as in \eqref{atou} by setting $s = t$ to obtain the uniform bound for $u_t(x)$.
\end{proof}

%\begin{comment}

The next result shows that the long-term behaviour of $u_t/\varphi$ and $a(t)$ are the same, which will be key to obtaining the correct constants in the bounds obtained in the previous lemma.

\begin{lemma}
\label{lem: iter}
Suppose that there exist $\kappa, \eta \in(0, \infty)$ such that $\norm{u_{t}}_{\infty} \le \kappa t^{-\eta}$ for all $t>0$. Then we can find some constant $C\in (0, \infty)$ which does not depend on $\kappa$  such that 
\[
\sup_{x \in E} \left|\frac{u_t(x)}{\varphi(x)} - a(t) \right|  \le C\kappa^{2} t^{-2\eta}, \quad \text{for all } t>0. 
\]
\end{lemma}

\begin{proof}
Comparing \eqref{eq: diff-eq} with \eqref{eq: diff-eq'}, we find that
\begin{align*}
\left|\frac{u_{t}(x)}{\varphi(x)}-a(t)\right|&\le \left| \frac{\psi_{t}[\mathbf 1](x)}{\varphi(x)}-\langle \tilde\varphi, \mathbf 1\rangle \right| + \int_{0}^{t}\left|\frac{\psi_{t-s}\big[\G[u_{s}]\big](x)}{\varphi(x)}-\langle \tilde\varphi, \G[u_{s}]\rangle\right| \d s\\
&\le C_{1} {\rm e}^{-\varepsilon t}+\int_{0}^{t} C_{1}{\rm e}^{-\varepsilon (t-s)} \norm{\G[u_{s}]}_{\infty} \,\d s,
\end{align*}
where the constant $\varepsilon>0$ exists thanks to (G2). Since we assumed that $\norm{u_{t}}_{\infty} \le \kappa t^{-\eta}$ for some constants $\kappa, \eta$, we can find $t_{0}>0$ such that $\textstyle\sup_{x \in E}\sup_{t\ge t_{0}}u_{t}(x)\le 1/2$. Take $t\ge 2t_{0}$, then the integral above can be bounded as follows:
\begin{align}
\int_{0}^{t} C_{1}{\rm e}^{-\varepsilon (t-s)} \norm{\G[u_{s}]}_{\infty} \,\d s&=\int_{0}^{t/2} C_{1}{\rm e}^{-\varepsilon (t-s)} \norm{\G[u_{s}]}_{\infty} \,\d s+\int_{t/2}^{t} C_{1}{\rm e}^{-\varepsilon (t-s)} \norm{\G[u_{s}]}_{\infty} \,\d s \notag\\
&\le \frac{C_{1}\norm{\gamma}_{\infty}N_{\max}}{\varepsilon}{\rm e}^{-\varepsilon t/2}+\int_{t/2}^{t} C_{1}{\rm e}^{-\varepsilon (t-s)} \norm{\G[u_{s}]}_{\infty} \,\d s,
\label{2nd}
\end{align}
where we have used Lemma \ref{lem: prop-G} (i). We note that, due to (G4), ${\V}[u_{s}]\le N_{\max}^{2}\norm{u_{s}}_{\infty}^{2}$. Using this together with Lemma \ref{lem: prop-G} (ii), we deduce for the second term in \eqref{2nd} that $\norm{\G[u_{s}]}_{\infty}\le \norm{{\V}[u_{s}]}_{\infty}+\norm{u_{s}^{3}}_{\infty}\le C_{2}\norm{u_{s}}_{\infty}^{2}$. The latter is bounded by $C_{2}\kappa^{2}s^{-2\eta}$ due to the assumption of the lemma. Therefore,
\begin{align*}
\int_{t/2}^{t} C_{1}{\rm e}^{-\varepsilon (t-s)} \norm{\G[u_{s}]}_{\infty} \,\d s &\le C_{3}\kappa^{2}\int_{t/2}^{t} {\rm e}^{-\varepsilon (t-s)}s^{-2\eta}\d s\\
&=\frac{C_{3}\kappa^{2}}{\varepsilon}\left(t^{-2\eta}-4{\rm e}^{-\varepsilon t/2}t^{-2\eta}\right)+C_{4}\int_{t/2}^{t}{\rm e}^{-\varepsilon (t-s)}s^{-2\eta-1}\d s=O(t^{-2\eta}). 
\end{align*}
Putting the pieces  together, we obtain the claimed bound in the lemma. 
\end{proof}

We are now ready to prove our Theorem \ref{lineardecaygen}

\begin{proof}[Proof of Theorem \ref{lineardecaygen}]
Applying Lemma \ref{lem: iter} with $\eta = 1$ and $\kappa$ being some positive constant (this is permitted thanks to Lemma \ref{lem: unif-small}), we have 
\[
\sup_{x \in E}\left|\frac{u_{t}(x)}{\varphi(x)}-a(t)\right|=O(t^{-2}), \quad t\to\infty. 
\]
On the other hand, we've seen in Lemma \ref{lem: a-lowerbd} that $a(t)^{-1}=O(t)$. It follows that 
\begin{equation}
\label{eq: u_t-a_t}
\sup_{x \in E}\left|\frac{u_{t}(x)}{\varphi(x)a(t)}-1\right|=O(t^{-1}), \quad t\to\infty. 
\end{equation}
Applying Lemma \ref{lem: prop-G} (iv), we deduce that
\begin{align}\notag
\sup_{x \in E}\left|{\V}[u_{t}](x)-a(t)^{2}{\V}[\varphi](x)]\right|&=\sup_{r\in D, \up\in V}a(t)^{2}\left|{\V}\Big[\frac{u_{t}}{a(t)}\Big](x)-{\V}[\varphi](x)\right| \\ \label{eq: bd-V}
&\le C\, a(t)^{2} \sup_{x \in E}\left|\frac{u_{t}(x)}{a(t)}-\varphi(x)\right| =O(t^{-3}). 
\end{align}
Therefore, for all $t\ge t_{0}$,
\begin{align*}
a(t)-a(t_{0})&=\int_{t_{0}}^{t} \langle \tilde\varphi, \G[u_{s}]\rangle \d s \\
&%\overset{\eqref{eq: bd-gg}}{=}
= -\frac12\int_{t_{0}}^{t} \Big(\langle \tilde\varphi, \gamma{\V}[u_{s}]\rangle +O(\norm{u_{s}}^{3}_{\infty})\Big) \d s \\
&%\overset{\eqref{eq: a-upperbd}}{=} 
=-\frac12\int_{t_{0}}^{t} \Big(\langle \tilde\varphi, \gamma{\V}[u_{s}]\rangle +O(s^{-3})\Big) \d s \\
&%\overset{\eqref{eq: bd-V}}{=}
= -\frac12\int_{t_{0}}^{t} \Big(a(s)^{2}\langle \tilde\varphi, \gamma{\V}[\varphi]\rangle +O(s^{-3})\Big) \d s\\
&%\overset{\text{Lemma \ref{lem: a-lowerbd}}}{=} 
=-\frac{\Sigma}{2}\int_{t_{0}}^{t} a(s)^{2}\big(1+o(1)\big) \,\d s,
\end{align*}
where we have used \eqref{eq: bd-gg} in the second equality, \eqref{eq: a-upperbd} in the third, \eqref{eq: bd-V} in the fourth  and Lemma \ref{lem: a-lowerbd} in the final equality.
This implies that 
\[
a(t)\sim \frac{2}{\Sigma t} \text{ as }t\to\infty.
\]
 The desired asymptotic for $u_{t}$ then follows from \eqref{eq: u_t-a_t}. 
\end{proof}

\section{Step 3: Proof of Theorem \ref{Ygen}}\label{Sectstep3} 
Returning to \eqref{prelimit}, the conclusions of Theorem \ref{genmoments} and Theorem \ref{lineardecaygen} yield
\[
\lim_{t\to\infty}
\mathbb{E}_{\delta_{x}}[\langle \varphi, X_t\rangle^k|\zeta>t]  
=\lim_{t\to\infty} \frac{ \mathbb{E}^\varphi_{\delta_{x}}[\langle \varphi, X_t\rangle^{k-1}/ t^{k-1} ]}{ t \mathbb{E}^\varphi_{\delta_{x}}\left[1/\langle \varphi, X_t \rangle\right]} = \left(\Sigma/2\right)^{k}k!, \qquad k\in\mathbb{N}.
\]
The right-hand side above is precisely the $k^{th}$ moment of an exponential distribution with parameter $2/\Sigma$. Hence, at this stage, again thanks to the {\it moment problem} for the exponential distribution, noting that $\langle\tilde\varphi, \varphi \rangle = 1$, we conclude that Theorem \ref{Ygen} holds for the special case that $f = \varphi$.

\medskip

We now deal more generally with the case that $f\in L^+_\infty(E)$. 

\begin{proof}[Proof of Theorem \ref{Ygen}]
As alluded to above, we introduce $\tilde{f} = f- \langle \tilde\varphi, f\rangle \varphi$ and consider the behaviour of $\langle \tilde{f}, X_t \rangle/ t$ as $t\to\infty$. The first thing we should note in this direction is that 
\[
\langle\tilde\varphi, \tilde f\rangle = \langle\tilde\varphi,f\rangle - \langle \tilde\varphi, f\rangle \langle\tilde\varphi,\varphi\rangle = 0,
\]
where we have used that $\langle\tilde\varphi, \varphi \rangle  = 1$.
Hence, 
assumption (G2) tells us that for $t$ sufficiently large
\[
\sup_{x \in E} \left| \frac{\psi_t[\tilde{f}] (x)}{\varphi(x)}  \right|  = O({\rm e}^{-\epsilon t}),
\]
for some $\epsilon>0$. In particular, we have $\sup_{x \in E}\psi_t[\tilde{f}](x) \to 0$ as $t \to \infty$.
\medskip

We also note that, under the assumption that $0\leq f\leq \varphi $ we have the estimate 
\begin{equation}
\label{bounds}
-\langle\tilde\varphi, f \rangle \varphi \leq \tilde f\leq (1-\langle\tilde\varphi , f\rangle)\varphi.
\end{equation}
Further, since $f\leq \varphi$,  $\langle\tilde \varphi, f  \rangle \leq \langle\tilde \varphi, \varphi\rangle = 1 $, the upper bound in \eqref{bounds} is positive. As such, $|\tilde f|\in L^+_\infty(E)$.

\medskip

We will use these facts to prove that $\langle \tilde f, X_t\rangle/t$ converges weakly under $\mathbb{P}_{\delta_{x}}(\, \cdot\, | \zeta>t)$ to zero as $t\to\infty$. In that case, thanks again to Slutsky's Theorem we can deduce that, for $x \in E$,
\begin{align}
&
\lim_{t\to\infty}\mathbb{E}_{\delta_{x}}\left[\left. \exp\left( -\theta\frac{ \langle f, X_t \rangle}{t}\right)\right|\zeta>t\right] \notag\\
& =\lim_{t\to\infty}\mathbb{E}_{\delta_{x}}\left[\left. \exp\left( -\theta \frac{\langle \tilde f, X_t\rangle}{t} -\theta \langle\tilde\varphi, f\rangle\frac{\langle\varphi, X_t\rangle}{t}\right)\right|\zeta>t\right]  \notag\\
&= \frac{1}{1 + \Sigma  \theta\langle\tilde\varphi, f \rangle/2},
\label{fullYaglom}
\end{align}
as required. 

\medskip

Let us thus consider the estimate
\begin{equation}
\mathbb{P}_{\delta_{x}}(|\langle \tilde{f}, X_t \rangle| /t> \varepsilon |\zeta>t)\leq \frac{1}{\varepsilon^2} \frac{ \mathbb{E}_{\delta_{x}}\left[ \langle \tilde{f}, X_t \rangle^2 /t\right]}{t\mathbb{P}_{\delta_{x}}(\langle 1, X_t\rangle >0) }.
\label{L2estimate2}
\end{equation}
In order to deal with the numerator on the righthand side above, we will use the {\it many-to-two} formula given in Lemma \ref{lem:M22}. We restate it here in the required form for convenience. 

\medskip

Let $g\in L^{+}_{\infty}(E)$. Then the many-to-two formula reads 
\begin{align}
\mathbb{E}_{\delta_x}\Big[\langle g, X_{t}\rangle^2 \Big]
&= \psi_{t}[g^2](x)
+\int_{0}^{t} \psi_{s}\Big[\gamma{\V}[\psi_{t-s}[g]] \Big](x)\d s,
\label{M22}
\end{align}
for $x \in E$ and $t\geq 0$, where ${\V}$ was given in \eqref{2point}.

\medskip

Hence, setting $g = \tilde f$ and dividing by $t^2$, we note that the first term on the right-hand side of \eqref{M22} has the behaviour 
\[
\lim_{t\to\infty} \frac{1}{t}\sup_{x \in E}\left|\psi_t[\tilde f^2](r,\up)\right| \leq  \lim_{t\to\infty} \frac{1}{t}\sup_{x \in E}\left|\psi_t[\tilde{f}^2](x) - \langle\tilde\varphi, \tilde{f}^2\rangle\varphi\right|
+ \frac{1}{t} \langle\tilde\varphi, \tilde f^2\rangle \norm{\varphi}_\infty  = 0,
\]
where the final equality uses (G2). For the integral term in \eqref{M22}, we can handle it using the ergodic convolution Theorem \ref{Ylem}, the change of measure Lemma \ref{spinemarkov} and the many-to-one formula in Lemma \ref{lem:M21}. In particular,  for $u\in[0,1]$ and $t\geq 0$,
\begin{align*}
 \psi_{ut}\Big[\gamma {\V}[\psi_{t(1-u)}[{\tilde f}] ] \Big](x) 
 &= \frac{\psi_{ut}\Big[\varphi \frac{\gamma}{\varphi}{\V}[\psi_{t(1-u)}[{\tilde f}] ] \Big](x)}{\varphi(x)} \varphi(x)\\
 &=\mathbf{E}^\varphi_{x} \left[\frac{\gamma(Y_{ut})}{\varphi(Y_{ut})}{\V}[\psi_{t(1-u)}[{\tilde f}]](Y_{ut})\right]\varphi(x).
\end{align*}
Moreover, defining $F(x, u, t) =\varphi(x)^{-1}\gamma(r, \up){\V}[\psi_{t(1-u)}[{\tilde f}] ](x)$, Lemma \ref{lem: prop-G} and the fact that $\psi_t[\tilde{f}] \to 0$ imply that
\[
F(x, u): = \lim_{t\to\infty} F(x, u, t)= \lim_{t\to\infty} \frac{ \gamma(x)}{\varphi(x)} {\V}[\psi_{t(1-u)}[{\tilde f}] ](x)=0,
\]
so that the conditions of Theorem \ref{Ylem} are trivially met (i.e. $F\in \mathcal{C}$). 
Hence, we have 
\begin{align*}
\lim_{t\to\infty}\mathbb E_{\delta_x}\left[\frac{\langle \tilde f, X_{t}\rangle^2}{t} \right]
&=\lim_{t\to\infty} \int_0^1 \mathbf{E}^\varphi_{x} \left[\frac{\gamma(Y_{ut})}{\varphi(Y_{ut})}{\V}[\psi_{t(1-u)}[{\tilde f}](Y_{ut})\right] \d u 
=0.
\end{align*}
Using this back in \eqref{L2estimate2}, together with the convergence of $t\mathbb{P}_{\delta_{x}}(\zeta >t) $ as $t\to\infty$ we see that 
\[
\lim_{t\to\infty}\mathbb{P}_{\delta_{x}}(| \langle \tilde{f}, X_t \rangle| /t> \varepsilon | \zeta>t)
\leq \lim_{t \to \infty}\frac{1}{\varepsilon^2} \frac{\mathbb{E}_{\delta_{x}}\left[ {\langle \tilde{f}, X_t \rangle^2 }/{t}\right]}{t\mathbb{P}_{\delta_{x}}(\zeta >t)} = 0.
\label{L2estimate}
\]
This tells us that $\langle \tilde{f}, X_t \rangle /t$ converges weakly to zero as $t\to\infty$ and hence, thanks to \eqref{fullYaglom}, this completes the proof of the Yaglom limit for general $f\in L^+_\infty(E)$. 
\end{proof}

\section{Neutron transport equation}\label{NTE}

We will now illustrate our main results in the context of the neutron branching process (NBP), for which the non-local branching assumption is crucial. Moreover, we are  interested in this class of branching process given the recent surge of results; see \cite{SNTE, SNTEII, SNTEIII, MCNTE, MultiNTE}. Let us spend some time discussing the underlying model. We will often reuse previously defined notation but in a way that will cause no confusion.

\smallskip

The neutron transport equation (NTE) describes the density of neutrons in an inhomogeneous fissile medium. The backwards form of the NTE can be written as follows\footnote{The operator $\nabla_r$ corresponds to gradient with respect to the variable $r\in D$.},
\begin{align}
\frac{\partial}{\partial t}\psi_t(r, \up) &=\up\cdot\nabla_r\psi_t(r, \up)  -\sigma(r, \up)\psi_t(r, \up)\notag\\
&+ \sigma_{\texttt{s}}(r, \up)\int_{V}\psi_t(r, \up') \pi_{\texttt{s}}(r, \up, \up')\d\up' + \sigma_{\texttt{f}}(r, \up) \int_{V}\psi_t(r, \up') \pi_{\texttt{f}}(r, \up, \up')\d\up',
\label{bNTE}
\end{align}
where the neutron density $\psi_t(r, \up)$ is a function of time, $t \ge 0$, and the configuration variables $ (r, \up) \in  D \times V$ where the spatial domain $D\subseteq\mathbb{R}^3$ is non-empty, smooth and bounded such that $\partial D$ has zero Lebesgue measure, and $V = \{\up\in \mathbb{R}^3:\up_{\texttt{min}}\leq  |\up|\leq \up_{\texttt{max}}\}$ is the velocity space, where $0<\upsilon_{\texttt{min}}<\upsilon_{\texttt{max}}<\infty$. Furthermore, the other components of~\eqref{bNTE} have the following interpretation:
\begin{align*}
\sigma_{\texttt{s}}(r, \up) &: \text{ the rate at which scattering occurs from incoming velocity $\up$,}\\
\sigma_{\texttt{f}}(r, \up) &: \text{  the rate at which fission occurs from incoming velocity $\up$,}\\
\sigma(r, \up) &: \text{ the sum of the rates } \sigma_{\texttt{f}}+ \sigma_{\texttt{s}} \text{ and is known as the total cross section,}\\
\pi_{\texttt{s}}(r, \up, \up')\d\up' &: \text{  the scattering yield at velocity $\up'$ from incoming velocity }  \up, \\
 &\hspace{0.5cm}\text{ satisfying }\textstyle{\int_V}\pi_{\texttt{s}}(r, \up, \up'){\rm d}\up'=1,\text{ and }\\
 \pi_{\texttt{f}}(r, \up, \up')\d\up' &:  \text{  the neutron yield at velocity $\up'$ from fission with incoming velocity }   \up,\\
 &\hspace{0.5cm}\text{ satisfying }{\color{black} \textstyle{\int_V}\pi_{\texttt{f}}(r, \up, \up')\d\up' <\infty.}
 \end{align*}
%An important assumption we will enforce throughout is that:
%\begin{center}
%{\bf   (H1): The quantities $\sigma_{\texttt{s}}$, $\sigma_{\texttt{f}}$, $\pi_{\texttt{s}}$ and $\pi_{\texttt{f}} $ are uniformly bounded away from  zero and infinity.}
%\end{center}
We also enforce the following initial and boundary conditions
\begin{equation}
\left\{
\begin{array}{ll}
\psi_0(r, \up) = g(r, \up) &\text{ for }r\in D, \up\in{V},
\\
&
\\
\psi_t(r, \up) = 0& \text{ for } t \ge 0 \text{ and } r\in \partial D
\text{ if }\up
\cdot{\bf n}_r>0,
\end{array}
\right.
\label{BC}
\end{equation}

where ${\bf n}_r$ is the outward unit normal at $r \in \partial D$ and $g: D \times V \to [0, \infty)$ is a bounded, measurable function. 

\medskip

%Throughout we will rely on the following assumptions in some (but not all) of our results:

%{\bf 
%{\color{black} 
%
%\begin{itemize}
%\item[(H1):]  The cross-sections $\sigma_{\texttt{s}}$, $\sigma_{\texttt{f}}$, $\pi_{\texttt{s}}$ and $\pi_{\texttt{f}} $ are uniformly bounded away from   infinity.
%
%
%\item[(H2):] 
%We have 
%$
%\sigma_{\texttt{s}} \pi_{\texttt{s}}  + 
%\sigma_{\texttt{f}} \pi_{\texttt{f}}>0$ on $D\times V\times V$.
%
%
%\item[(H3):]  There is an open ball $B$ compactly embedded in $D$ such that $\sigma_{\texttt{f}}\pi_{\texttt{f}} >0$ on $B\times V\times V$.
%
%
%\item[(H4):]  the fission offspring are bounded in number  by the constant $N_{\max}> 1$.
%
%\end{itemize}
%}
%}
%
%
%Note, the assumption (H1) ensures that all activity occurs at a maximum rate. Assumption (H2) ensures that at least some activity occurs, whether it be scattering or fission; together with (H3), it ensures that there is at least some fission as well as scattering. Finally (H4) is a physical constraint that is natural to nuclear fission, typically no more than 3 neutrons are produced during an average fission event. Figure \ref{fig} illustrates the complex nature of the inhomogeneity in the domain one typically considers.
\begin{wrapfigure}[16]{r}{0.5\textwidth}
\vspace{-20pt}
\label{fig}
  \begin{center}
\includegraphics[height=6cm]{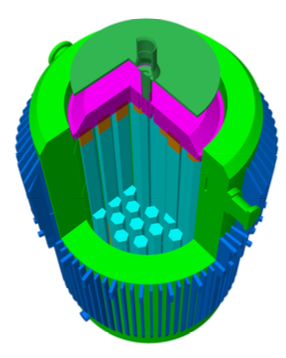}
  \end{center}
  \vspace{-20pt}
\caption{\it The geometry of a nuclear reactor core representing a physical domain $D$, onto which the different cross-sections $ \sigma_{\emph{\texttt{s}}},  \sigma_{\emph{\texttt{f}}},  \pi_{\emph{\texttt{s}}},  \pi_{\emph{\texttt{f}}}$ are mapped, also as a function of neutron velocity.}
%\vspace{20pt}
\end{wrapfigure}
Recent literature has explored the relationship between  the NTE and a class of stochastic processes; cf. \cite{MultiNTE, MCNTE, SNTE, SNTEII, SNTEIII}. Roughly speaking, the solution to the NTE can be seen as a semigroup of the underlying physical process of neutrons, described as a branching Markov process, i.e. the NBP. The NBP defined in \cite{SNTE, SNTEII, MultiNTE}, which at time $t\ge 0$ is represented by a configuration of particles which are specified via their physical location and velocity in $D\times V$, say $\{(r_i(t), \up_i(t)): i = 1,\dots , N_t\}$, where $N_t$ is the number of particles alive at time $t$. Analogously to \eqref{empdist} in the general MBP setting, the NBP, which we also denote $(X_t, t \ge 0)$, is given by the empirical distribution of these configurations, with $E = D \times V$ and $x_i(t) = (r_i(t), \up_i(t))$.

\medskip

Thus, $(X_t, t\geq 0)$ is a stochastic process valued in the space of finite counting measures
$
\mathcal{M}(D\times V): = \{\textstyle{\sum_{i = 1}^n}\delta_{(r_i,\up_i)}: n\in \mathbb{N}, (r_i,\up_i)\in D\times V, i = 1,\cdots, n\}
$
which evolves randomly as follows.

\medskip

A particle positioned at $r$ with velocity $\up$ will continue to move along the trajectory $r + \up t$, until one of the following things happens. 
\medskip

\noindent{\bf(i)} The particle leaves the physical domain $D$, in which case it is instantaneously killed. 
\medskip

\noindent{\bf(ii)} Independently of all other neutrons, a scattering event occurs when a neutron comes in close proximity to an atomic nucleus and, accordingly, makes an instantaneous change of velocity. For a neutron in the system with position and velocity $(r,\up)$, if we write 
$T_{\texttt{s}}$ for the random time that scattering may occur then, independently of any other physical event that may affect the neutron, 
$
\Pr(T_{\texttt{s}}>t) = \exp\{-\textstyle{\int_0^t} \sigma_{\texttt{s}}(r+\up s, \up){\rm d}s \}, $ for $t\geq0.$

\medskip
When scattering occurs at space-velocity $(r,\up)$, the new velocity $\up'\in V$ is selected with probability $\pi_{\texttt{s}}(r, \up, \up')\d\up'$. 
\medskip

\begin{wrapfigure}{r}{0.3\textwidth}
\vspace{10pt}
  \begin{center}
\includegraphics[height=5cm]{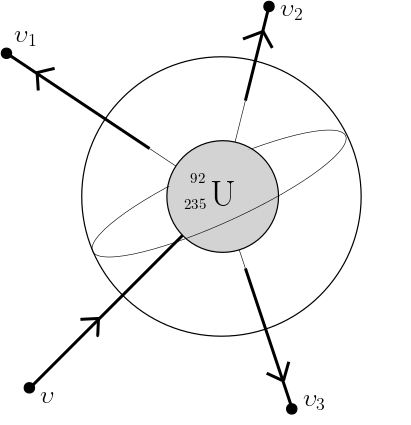}
  \end{center}
  \vspace{-15pt}
\caption{\it A fission event between e.g. a Uranium-235 atom and a neutron with incoming velocity $\upsilon$ produces three outgoing neutrons with velocity $\upsilon_1,\upsilon_2, \upsilon_3$.}
\vspace{-40pt}
\label{u235}
\end{wrapfigure}

\noindent{\bf(iii)}  Independently of all other neutrons, a fission event occurs when a neutron smashes into an atomic nucleus. 
For a neutron in the system  with initial position and velocity $(r,\up)$, if we write $T_{\texttt{f}}$ for the random time that fission may occur, then, independently of any other physical event that may affect the neutron, 
$
\Pr(T_{\texttt{f}}>t) = \exp\{-\textstyle{\int_0^t} \sigma_{\texttt{f}}(r+\up s, \up){\rm d}s \},$ for $t\geq 0.
$
\medskip

When fission occurs, the smashing of the atomic nucleus produces lower mass isotopes and releases a random number of neutrons, say $N\geq 0$, which are ejected from the point of impact with randomly distributed, and possibly correlated, velocities, say $\{\up_i: i=1,\cdots, N\}$. See Figure \ref{u235}. The outgoing velocities are described by  the atomic random measure 
\begin{equation}
\label{PP}
\mathcal{Z}(A): = \sum_{i= 1}^{N } \delta_{\up_i}(A), \qquad A\in\mathcal{B}(V).
\end{equation}

If such an event occurs at location $r\in D$ from a particle with incoming velocity $\up\in{V}$, 
we denote by ${\mathcal P}_{(r,\up)}$ the law of $\mathcal{Z}$, and by $\mathcal{E}_{(r, \up)}$ the corresponding expectation operator.
The probabilities ${\mathcal P}_{(r,\up)}$ are such that, for bounded and measurable $g: V\to[0,\infty)$,
\begin{align}
\int_V g(\up')\pi_{\texttt{f}}(r, v, \up')\d\up' &= {\mathcal E}_{(r,\up)}\left[\int_V g(\up')\mathcal{Z}(\d \up')\right] %=: {\mathcal E}_{(r,\up)}[\langle g, \mathcal{Z}\rangle].
\label{Erv}
\end{align}
%Here, we are using the notation $\langle f, g\rangle$ to mean $\textstyle \int_{D\times V}f(r,\up)g(r,\up)\d r\d \up$.
Note, the possibility that $\Pr(N = 0)>0$, which is tantamount to neutron capture (that is, where a neutron slams into a nucleus but no fission results and the neutron is absorbed into the nucleus), is included in the notion of fission. Thus, the non-locality occurs in the velocity variable. To make the connection between \eqref{PP} and \eqref{PPgen}, we set $x_i = (r, \up_i)$, $i = 1, \dots, n$.

\medskip

%We will tend to work with the process $X$ by describing the behaviour of $\langle f, X_t\rangle$, $t\geq 0$, for 
%$f\in L^+_\infty(D\times V)$, the space of non-negative, uniformly bounded functions on $D\times V$, where 
%$\textstyle \langle f, X_t\rangle : =  \int_{D\times V}f(r,\up)X(\d r, \d \up)= \sum_{i = 1}^{N_t} f(r_i(t), \up_i(t))$. Having introduced this inner product notation, we note that we will immediately abuse it by writing e.g. $\langle f, g\rangle$ to mean $\textstyle \int_{D\times V}f(r,\up)g(r,\up)\d r \d \up$, as well as $\textstyle\langle g, \mathcal{Z}\rangle$ to mean $\textstyle\int_V g(\up)\mathcal{Z}(\d \up)=\sum_{i=1}^N g(\up_i)$. In the latter case, we can otherwise think of the point process $\mathcal{Z}$ in \eqref{Erv} as placing offspring at $\{(r_i, \up_i), i = 1,\cdots, N\}$ with $r_i = r$, $i = 1,\cdots, N$ under $\mathcal{P}_{(r,\up)}$.
%
%
%\medskip

%It is worth noting that the NBP has a non-local branching mechanism given by 
%\[
%G[f](r,\up): = \sigma_{\texttt f}(r,\up)\left(\mathcal{E}_{(r,\up)} \left(\prod_{i =1}^Nf(r,\up_i)\right)- f(r,\up)\right), \qquad f\in L^+_\infty(D\times V).
%\]
%The non-locality pertains to the random assignment of velocities (in $V$) of offspring.

The NBP is thus parameterised by the quantities $\sigma_{\texttt s}(r,\up), \pi_{\texttt s}(r,\up, \up'), \sigma_{\texttt f}(r,\up)$ and the  measures ${\mathcal P}_{(r,\up)}$ for $r\in D,\up\in V$. Accordingly, we refer to it as a $(\sigma_{\texttt s}, \pi_{\texttt s}, \sigma_{\texttt f}, \mathcal{P})$-NBP.  Although a $(\sigma_{\texttt s}, \pi_{\texttt s}, \sigma_{\texttt f}, \mathcal{P})$-NBP is uniquely defined, a NBP specified by $(\sigma_{\texttt s}, \pi_{\texttt s}, \sigma_{\texttt f}, \pi_{\texttt f})$ alone is not.  Nonetheless, it is easy to show that for a given $\pi_{\texttt f}$,  a $(\sigma_{\texttt s}, \pi_{\texttt s}, \sigma_{\texttt f}, \mathcal{P})$-NBP satisfying \eqref{Erv} always exists. See the discussion in Section 2 of \cite{SNTE}.

\medskip

%In order to understand the relationship of the Yaglom limit for the NBP with the NTE, we need to introduce a more appropriate representation of the  the NTE  \eqref{bNTE}.

To make the link between the NTE and the NBP clear, we will reformulate the NTE \eqref{bNTE} in terms of its{ \it Duhamel} representation. From a probabilistic perspective, this is a more natural way of writing the NTE since the irregular nature of the gradient operator makes it hard to work \eqref{bNTE} the pointwise sense. 
The following integral equation was introduced in \cite{SNTE} as a mild form of the NTE,
\begin{equation}
\psi_t[g] = {\texttt{U}}_t[g] + \int_0^t {\texttt{U}}_s[({\bS} + {\bF})\psi_{t-s}[g]]\d s, \qquad t\geq 0,
\label{mild}
\end{equation}
where the semigroup $(\mathtt{U}_t, t\geq0)$ is defined by
\begin{equation*}
{\U}_t[g](r, \upsilon) = g(r + \upsilon t, \upsilon)\mathbf{1}_{(t<\kappa_{r, \upsilon}^D)},\quad 
\kappa_{r, \upsilon}^D = \inf\{t>0 : r+ \up t \not\in D\},
%\label{advhat}
\end{equation*}
 the operator $\bS$ has action on $L^+_\infty(D\times V)$, the space of non-negative and uniformly bounded measurable functions on $D\times V$, given by 
\[
\bS g(r,\up) =\sigma_{\texttt s}(r,\upsilon) \int_{D\times V} [g(r,\upsilon')  - g(r,\up)]\pi_{\texttt s} (r, \upsilon, \upsilon')\d \upsilon', \qquad g\in L^+_\infty(D\times V),
\]
and, similarly,  $\bF$ has action 
\[
\bF g(r, \up) =\sigma_{\texttt f}(r,\upsilon) \int_{D\times V} g(r,\upsilon') \pi_{\texttt f} (r, \upsilon, \upsilon')\d \upsilon' - \sigma_{\texttt f}(r,\upsilon) g(r,\up),  \qquad g\in L^+_\infty(D\times V).
\]

We encourage the reader to compare \eqref{mild} with \eqref{lineareq}. Indeed,  in this setting, the Markov motion is just a deterministic straight line in between scattering and fission events. In particular, $\bP_t = \U_t$ and $\mathscr{A} = \bS + \bF$. Noting ${\mathtt T}: = \up\cdot\nabla_r$ as the operator corresponding to the semigroup $(\U_t, t\geq 0)$, it is also possible to reformulate \eqref{mild} as
\[
  \psi_t[g] = \mathtt{Q}_t[g] + \int_0^t \mathtt{Q}_s[\bF[\psi_{t-s}[g]]] \d s,
\]
where $\mathtt{Q}$ is the Markov semigroup corresponding  to ${\mathtt T}+ \bS$. We refer the reader to \cite{SNTEII} for further details. 

\medskip

It was shown in~\cite{MultiNTE, SNTE, DL6, D} that under the assumption\footnote{We note that this result was actually proved with the additional assumption that $\sigma_\mathtt{s}\pi_\mathtt{s} + \sigma_\mathtt{f}\pi_\mathtt{f}$ {is uniformly bounded below} on $D \times V \times V$, however this is not actually needed.}
\medskip
{\bf
\begin{itemize}
\item[(H1):]  The cross-sections $\sigma_{\texttt{s}}$, $\sigma_{\texttt{f}}$, $\pi_{\texttt{s}}$ and $\pi_{\texttt{f}} $ are uniformly bounded away from infinity
\end{itemize}
}
\medskip
for $g\in L^+_\infty(D\times V)$ the following linear semigroup of the NBP, 
%there is a solution to \eqref{mild} which is bounded in the sense that there exist  constants $C_1,C_2>0$ such that  $\psi_t[g](r,\up)\leq  C_1\exp(C_2 t)$, uniformly for all $t\geq 0$. 
\begin{equation}
\psi_t[g](r,\up) : = \mathbb{E}_{\delta_{(r, \up)}}[\langle g, X_t \rangle], \qquad t\geq 0, r\in \bar{D}, \up\in{V},
\label{semigroup}
\end{equation}
solves \eqref{mild}. Note that (H1) implies (G1) in this setting, since it implies that 
\[
m(r, \up) = \mathcal{E}_{(r, \up)}[\langle 1, \mathcal{Z}\rangle] = \int_V \pi_\mathtt{f}(r, \up, \up')\d \up'
\] 
is uniformly bounded from above.

\medskip

As shown in \cite{MultiNTE, SNTE}, there is a second stochastic process whose average behaviour also provides a solution to the NTE, called the neutron random walk (NRW). 

\medskip

Define, for $r\in D$, $\up,\up'\in V$,
\begin{equation}
\alpha(r,\up)\pi(r,\up,\up') = 
\sigma_{\texttt{s}}(r,\up)\pi_{\texttt{s}}(r, \up, \up') + \sigma_{\texttt{f}}(r,\up) \pi_{\texttt{f}}(r, \up, \up').%\qquad r\in D, \up,\up'\in V.
\label{spineap0}
\end{equation}

Then the $\alpha \pi$-NRW can be defined as the process $((R_t, \Upsilon_t), t \ge 0)$ whose pathwise evolution is such that, if its current position and velocity are $R_0 = r\in D$ and $\Upsilon_0 = \up\in V$, respectively, then the process will move forward in time in $D$ with linear motion $r+ \up t$.  Unless the NRW first exits the domain $D$, at rate $\alpha(r + \up t, \up)$, a scattering event occurs. Suppose $r'\in D$  is the NRW's position at the scattering event. Then a new velocity $\up'$ is chosen with probability $\pi(r',\up,\up')\d \up'$. Thus the process evolves by moving in straight lines between scatter times, until it exits the domain $D$. 
Note that in order to specify $\alpha$ and $\pi$ individually, it is sufficient to define the product $\alpha\pi$ since we require $\pi$ to be a probability distribution.

\medskip

Suppose we  denote the probabilities of the $\alpha \pi $-NRW by $\mathbf{P}   : =(\mathbf{P} _{(r,\up)}, (r,  \up)\in  D\times V)$. 
Moreover, for $r\in D$ and $\up\in V$, let 
\begin{equation}
  \beta(r,\up)=\sigma_{\texttt{f}}(r,\up)\left(\int_V\pi_{\texttt{f}}(r, \up,\up')\d\up'-1\right)\geq -\sup_{r\in D, \up\in V}\sigma_{\texttt{f}}(r,\up)>-\infty,
  \label{betadef}
\end{equation}
where the lower bound is due to assumption (H1).
The following lemma\footnote{Again, this result was actually proved with the additional assumption that $\sigma_\mathtt{s}\pi_\mathtt{s} + \sigma_\mathtt{f}\pi_\mathtt{f}$ is uniformly bounded below on $D \times V \times V$, however this is not actually needed.}, taken from \cite{SNTE}, identifies the law $\mathbf{P} $ in terms of the solution to the NTE \eqref{mild}, the so called {\it many-to-one} representation. Again, we encourage the reader to compare this to the general case presented in Lemma \ref{lem:M21}.

 \begin{lemma}[Many-to-one]\label{M21}Under assumptions (H1), the  law $\mathbf{P}$ is related to NTE \eqref{mild} via a  semigroup representation, different to  \eqref{semigroup}, given by 
\begin{equation}
\psi_t[g](r,\up)  = \mathbf{E}_{(r,\up)}\left[{\rm e}^{\int_0^t\beta(R_s, \Upsilon_s)\D s}g(R_t, \Upsilon_t) \mathbf{1}_{\{t < \tau^D\}}\right], \qquad t\geq 0,r\in D, \up\in V,
\label{phi}
\end{equation}
where $\tau^D = \inf\{t>0 : R_t\not\in D\}$. 
\end{lemma}

\medskip

{As discussed in Remark \ref{criticalityremark}}, understanding the long-term growth of the NBP boils down to an understanding of the lead eigenvalue associated to the semigroup \eqref{semigroup} in the spirit of Perron-Frobenius theory. In classical neutron transport theory, this corresponds to the criticality of the underlying NTE \eqref{bNTE}, i.e. the sign of the leading eigenvalue associated to the latter.

\begin{theorem}\label{CVtheorem}
Suppose that (H1) holds, in addition to the condition 
\medskip

{\bf 
\begin{itemize}
\item[(H2):]  $\textstyle{\inf_{r\in D, \up, \up'\in V} \alpha(r,\up)\pi(r,\up,\up') >0.}$
\end{itemize}
}
\medskip

Then, for the semigroup  $(\psi_t,t\geq0)$ identified by \eqref{mild},  there exists  a $\lambda_*\in\mathbb{R}$, a  positive\footnote{To be precise, by a positive eigenfunction, we mean a mapping from $D\times V\to (0,\infty)$. This does not prevent it being valued zero on $\partial D$, as $D$ is an open bounded, convex domain.} right eigenfunction $\varphi \in L^+_\infty(D\times V)$ and a left eigenmeasure which is absolutely continuous with respect to Lebesgue measure on $D\times V$ with density $\tilde\varphi\in L^+_\infty(D\times V)$, both having associated eigenvalue ${\rm e}^{\lambda_* t}$, and such that $\varphi$  (resp. $\tilde\varphi$) is uniformly (resp. a.e. uniformly) bounded away from zero on each compactly embedded subset of $D\times V$. In particular, for all $g\in L^+_{\infty}(D\times V)$,
\begin{equation}
\langle\tilde\varphi, \psi_t[g]\rangle = {\rm e}^{\lambda_* t}\langle\tilde\varphi,g\rangle\quad  \text{(resp. } 
\psi_t[\varphi] = {\rm e}^{\lambda_* t}\varphi
\text{)} \quad t\ge 0.
\label{leftandright}
\end{equation}
Moreover, there exists $\varepsilon>0$ such that %for all $g\in L^+_\infty(D\times V)$,
\begin{equation}
\sup_{g\in L^+_{\infty, 1}(D\times V)}  \left\|{\rm e}^{-\lambda_* t}{\varphi}^{-1}{\psi_t[g]}-\langle\tilde\varphi, g\rangle\right\|_\infty = O({\rm e}^{-\varepsilon t})\, \text{ for all $t \ge 0$.}
\label{spectralexpsgp}
\end{equation}
%In particular, setting $g  \equiv 1$
%
%\begin{equation}
%  \left\|{\rm e}^{-\lambda_c t}\varphi^{-1}\mathbf{P}_{\cdot}(t < \emph{\texttt{k}}) -\langle\tilde\varphi, 1\rangle\right\|_\infty = O({\rm e}^{-\varepsilon t}) \text{ as $t\rightarrow\infty$.}
%\label{die}
%\end{equation}
\end{theorem}
{As noted in Remark \ref{criticalityremark}, }
the criticality of the NTE is determined by the sign of $\lambda_*$. When $\lambda_*>0$ (resp. $\lambda_*<0$), the NBP is supercritical (resp. subcritical), situations which are undesirable for nuclear reactors. The   case of criticality,   $\lambda_* = 0$, has the property that there is a mean steady-state for the NBP in the sense that $\psi_t[g](r,\up)\to \langle\tilde\varphi, g\rangle\varphi(r,\up)$ as $t\to\infty$. This is the setting in which nuclear reactors typically operate and thus fuels our motivation for studying the NBP at criticality. 

\medskip

Despite the convergence to the steady state in the critical case, in Theorem 5.2 of \cite{SNTE}, under the additional assumption 
\medskip

{\color{black}
{\bf  (H3): There exists an open ball $B$, compactly embedded in $D$, such that
$$
\inf_{r \in B, \upsilon, \upsilon' \in V}\sigma_{\texttt{f}}(r, \upsilon)\pi_{\texttt{f}}(r, \upsilon, \upsilon') > 0,
$$}

it was shown for the critical NBP that, for all $r\in D, \up\in V$,
 \[
 \mathbb{P}_{\delta_{(r,\up)}}(\zeta <\infty) = 1\,\text{ where }\,\zeta = \inf\{t>0: \langle 1, X_t\rangle = 0\},
 \]
i.e. the probability of extinction is $1$. 
}

\medskip

For the rest of this section, we restrict ourselves to the case where $\lambda_* = 0$. As such, in this case, Theorem~\ref{CVtheorem} implies that (G2) holds for the NBP. Moreover, due to the above discussion, (H3) entails (G3). We will also assume that (G4) holds, which we restate in this context:
\medskip
{\bf
\begin{itemize}
\item[(H4):] The number of offspring produced at a fission event is bounded above by a constant $N_{\max}$.
\end{itemize}
}

\medskip

We note that (H4) is a natural physical constraint in the setting of the NBP, since fission events can only result in the release of a finite number of neutrons. 

\medskip

Before giving the equivalent formulations of Theorems \ref{lineardecaygen} and \ref{Ygen} in the setting of the NBP, 
we will need to introduce additional notation and one more assumption.
For $g \in L_\infty^+(V)$ define the spatial variance of fission velocities by 
\begin{equation}
  {\VV}[g](r,\up) = \mathcal{E}_{(r,\up)}\left[\langle g ,\mathcal{Z}\rangle^2-\langle g^2 ,\mathcal{Z}\rangle \right]
  = \mathcal E_{(r, \up)}\Big[\sum_{\substack{1\le i,j\le N\\  i\ne j}}g(\up_{i})g(\up_{j})\Big].
  \label{Var}
\end{equation}
We will often abuse our notation and, for $g\in L^+_\infty(D\times V)$, write ${\VV}[g] $ in place of $ {\VV}[g(r, \cdot)]$. We will also redefine the key quantity
\begin{equation}
\Sigma =\langle \tilde\varphi , \sigma_{\mathtt f}{\VV}[{\varphi}]\rangle,
\label{Sigma}
\end{equation}
which is finite thanks to (H1), (H4) and the boundedness of $\tilde\varphi$ (cf. Theorem \ref{CVtheorem}).
%with the slight abuse of notation in that ${\VV}[\varphi] = {\VV}[\varphi(r, \cdot)]$.

\medskip

The following assumption can be seen as the NBP counterpart of (G5).

\medskip

%{\bf 
%\begin{itemize}
%  \item[(H5):] There exists a constant $c_1 > 0$ such that $\inf_{r\in D, \up\in V}\sigma_\mathtt{f}(r, \up) \ge c_1$.
%  \item[(H6):] There exists a function $\tilde{\pi}_\mathtt{f} : D \times V \times V \times V \to [0, \infty)$ such that, for $g(r,\cdot)\in L^+_\infty( V)$,
 %   \begin{align*}
 %     \mathcal{E}_{(r, \up)}\left[\mathbf{1}_{\{N \ge 2\}}\sum_{i \neq j}g(r, \up_i)g(r, \up_j)\right]
 %    & = \int_{V \times V} g(r, \up_1)g(r, \up_2)\tilde{\pi}_\mathtt{f}(r, \up, \up_1, \up_2)\d \up_1 \d \up_2\\
%     &\geq  c_2\int_{V \times V} g(r, \up_1)g(r, \up_2)\d \up_1 \d \up_2\
%    \end{align*}
%    and $\tilde{\pi}_\mathtt{f}(r, \up, \up_1, \up_2) \ge c_2$ for some constant $c_2 > 0$ and for all $r \in D$, $\up, \up_1, \up_2 \in V$.
%Now take $g(r,\up) =\theta(r)h(\up)$ where $\langle\tilde\varphi, \theta\rangle = \infty$
{\bf 
(H5): There exists a constant $C > 0$ such that for all $g \in L^+_\infty(D \times V)$, 
\[
  \langle \tilde\varphi, \sigma_\mathtt{f}{\VV}[g]\rangle \ge C\langle \tilde\varphi, \hat{g}^2\rangle,
\]
where $\textstyle \hat{g} : D \to [0, \infty) : r \mapsto \int_Vg(r, \up')\d \up'$.

}

\medskip

As with assumption (G5), (H5) can be thought of an irreducibility type condition on the fission operator. However, in this case, when $g = u_t$, (H5) implies that if there is a fission event in the stationary distribution, the process should have a better chance of survival than producing two particles with isotropic velocities. We note that this condition is satisfied in the setting of \cite{MWY} and in the isotropic case if we assume that, for all $r\in D, \up\in V$, $\mathcal{E}_{(r,\up)}[N(N-1)] \ge c > 0$ for some constant $c$, 

\medskip

We now restate the main results in the context of the NBP.

\smallskip

\begin{theorem}\label{lineardecay}Suppose that assumptions (H1), (H2), (H3), (H4) and (H5) hold and $\lambda_*=0$. Then, for all $r, \in D$ and $\up \in V$,
\[
\lim_{t\to\infty} t\mathbb{P}_{\delta_{(r,\up)}}(\zeta>t) = \frac{2\varphi(r, \up)}{\Sigma}.
\]
\end{theorem}

\medskip

\begin{theorem}[Yaglom limit for critical neutron transport]\label{Y} Suppose that assumptions (H1), (H2), (H3), (H4) and (H5) hold with $\lambda_*=0$. For $f\in L^+_\infty(D\times V)$ and for all $r\in D$ and $\up\in V$, 
\begin{equation}
\lim_{t\to\infty}\mathbb{E}_{\delta_{(r,\up)}}\left[\left. \exp\left( -\theta \frac{ \langle f, X_t \rangle}{t}\right)\right|\zeta>t\right] = \frac{1}{1 + \langle \tilde\varphi, f\rangle\Sigma  \theta/2}.
\label{Yaglom}
\end{equation}
\end{theorem}
\medskip

Due to the generality of our methods, the proofs of the above two theorems are almost identical to those given in the general MBP case. However, since (H5) is weaker than (G4) in the case where $E$ is bounded, the proof of 
Theorem \ref{lineardecay} requires some more work. More precisely, starting in the same way as in~\ref{Sectstep2} and using (H5), we obtain
\begin{align}
\notag
a'(t) =  \langle \tilde \varphi, \mathscr{G}[u_{t}]\rangle 
&\le -C_{1}\langle \tilde \varphi,  \sigma_\mathtt{f}\mathsf{V}\!\!\!\mathsf{V}[u_{t}]\rangle  \\ \label{bd: at1}
& \le - C_{2} \int_{D \times V}\tilde\varphi(r, \up)\int_{V}\int_V u_t(r, \up_1)u_t(r, \up_2)\d r\,\d\up\,\d\up_1\d\up_2.
\end{align}
Now, noting that  H\"{o}lder's inequality implies that 
\[
\left(\int_{D} f(r) \d r\right)^{3} \le C_{3}\int_{D} |f(r)|^{3} \d r, 
\] 
we can apply this to $r\mapsto \int_{V}\tilde\varphi(r, \up) u_t(r, \up)\d\up$ to obtain
\begin{align*}
  a(t)^3 &= \left(\int_{D \times V}\tilde\varphi(r, \up) u_t(r, \up)\d r\d\up \right)^3  \le C_3\int_D\left(\int_V\tilde\varphi(r,\up)u_t(r, \up)\d\up \right)^3\d r \\
  &= C_3\int_D \d r \int_{V \times V \times V}\tilde\varphi(r, \up)\tilde\varphi(r, \up_1)\tilde\varphi(r, \up_2)
  u_t(r, \up)u_t(r, \up_1)u_t(r, \up_2)\d\up\,\d\up_1\d\up_2 \\
  &\le C_4\norm{u_{t}}_{\infty} \int_{D \times V}\tilde\varphi(r, \up)\int_{V}\int_V u_t(r, \up_1)u_t(r, \up_2)\d r\,\d\up\,\d\up_1\d\up_2,
\end{align*}
since $\tilde\varphi$ is uniformly bounded. Comparing this with \eqref{bd: at1}, we have
\[
a'(t)  \le - \frac{C_{5}}{\norm{u_{t}}_{\infty}} a(t)^{3}, \quad t\ge t_{0}. 
\]
We have previously seen that $\norm{u_{t}}_{\infty}\to 0$, which implies that for any fixed $\epsilon>0$, there exists $t'_{0}=t'_{0}(\epsilon)$ so that 
\[
a'(t) \le  - \frac{C_{5}}{\epsilon^{2}} a(t)^{3}, \quad t\ge t_{1}:=\max(t_{0}, t'_{0}). 
\]
Concluding as in~\ref{Sectstep2}, we obtain coarse upper and lower bounds but now of the order $1/\sqrt{t}$. 
Then, Lemma \ref{lem: iter} allows us to bootstrap these bounds to obtain bounds of order $1/t$. From there, we can then conclude as in the general case. 

\medskip

As noted in the introduction, the work of Mori et al. \cite{MWY} gives the only known results in the NBP setting. In that  paper, the authors took  advantage of an isotropic assumption for scattering and the point process $\mathcal{Z}$, as well as homogeneous cross-sections, which allows the system to be treated as if it has local branching. More precisely, $\mathcal{Z}$ is made up of a random number of offspring, each of which are independent and uniformly distributed on $V$, which is otherwise taken to be the unit sphere $\mathbb{S}^2$.  The anisotropic behaviour of the fission we consider in this paper means that, unlike \cite{MWY},  we genuinely cannot treat branching as a local phenomenon, and thus require  more involved analysis, which may as well be delivered for general BMPs, as indeed we have done.

\section*{Acknowledgements} {We would like to thank an anonymous referee and the AE, who made a number of very helpful suggestions. We would also like to thank Ellen Powell for useful comments and} our industrial partners, specifically Prof. P. Smith and Dr. G. Dobson of the ANSWERS group from Jacobs, for the use of the picture in Figure \ref{fig}.

\appendix

\section{Combinatorial results}
Here we give some combinatorial results that are of importance to us. The first is the multinomial expansion formula, the second is a variant of the multinomial formula for sums of i.i.d. random variables. The third is an identity which does not appear to be easily available in the literature, but is nonetheless relatively straightforward to prove. 

\begin{lemma}[Multinomial identity] Suppose that $a_1,\cdots, a_n$ are real-valued and $k\in\mathbb{N}$. Then 
\begin{equation}
\left(\sum_{i = 1}^n a_i \right)^k  = \sum_{[k_1,\cdots, k_n]}{k \choose k_1, \cdots, k_n} \prod_{i =1}^n a_i^{k_i},
\label{ais}
\end{equation}
where the sum is over the set $[k_1,\cdots, k_n]$ of all combinations of non-negative $k_1, \dots, k_n$ such that $\textstyle \sum_{i= 1}^n k_i= k$,  which are assigned as power indices to $(a_1,\cdots, a_n)$.
\end{lemma}

The set $[k_1,\cdots, k_n]$ is very complex, particularly when one takes into account the fact that, for example, when $n>k$, some of the $k_i$ are necessarily zero. The combinatorics  become more complicated when we  apply this lemma to the setting that the $a_i$ are iid random variables and the left-hand side of \eqref{ais} is averaged. 

\begin{lemma}\label{iidmulti}
Suppose that $Y_1,\dots Y_n$ are iid random variables which are equal in distribution to $(Y,{\rm P})$, then%, for $n>k$,
\begin{equation}
{\rm E}\left[\left(\sum_{i = 1}^n Y_i \right)^k\right]
=\sum_{j = 1}^k 
%\frac{n!}{(n-j)!}
{n \choose j }{\mathbf{1}_{(j\leq n)}}
 \sum_{[k_1,\dots, k_j]_+}{k \choose k_1, \cdots, k_j} \prod_{i =1}^j {\rm E}[Y^{k_i}]
 \label{j!}
\end{equation}
{where the sum is over the set $[k_1,\dots, k_j]_+$ of all combinations of strictly positive $k_1,\dots ,k_j$ such that $\textstyle \sum_{i= 1}^j k_i= k$.% if $j\leq n$ and otherwise $[k_1,\cdots, k_j]_+ = \emptyset$.
} 
\end{lemma}

%\begin{rem}\rm
%It is worth remarking that the set $[k_1,\cdots, k_j]_+$ may be an empty set depending on whether $k<n$ on $k\geq n$. Indeed, in the latter case, then necessarily $[k_1,\cdots, k_j]_+=\emptyset$ for all $j\geq n$ as one cannot select more than $n$ different of the $Y_i$-s across the $k$ brackets in the product. 
%\end{rem}
\begin{proof}[Proof of Lemma \ref{iidmulti}]We start by re-examining \eqref{ais}.
To avoid having to  deal with  the zero values of $k_i$ in \eqref{ais} we would like to write 
{\color{black}
\[
[k_1,\cdots, k_n]  = \bigcup_{j = 1}^{k\wedge n} \bigcup_{\{i_1,\cdots,i_j\}\in \mathcal{I}_j}\bigcup_{\{K_{i_1},\cdots,K_{i_j}\}\in [k_1,\cdots, k_j]_+}
\{K_1,\cdots, K_n \}|_{K_\ell = 0, \ell \not\in \{i_1,\cdots,i_j\}}
\]
where the first union in $j$ identifies the number $K_i$ from $\{K_1,\cdots, K_n\}\in [k_1, \cdots, k_n]$  which are strictly positive, the second union  is over, $\mathcal{I}_j$, the set of of all possible choices of $j$ indices $\{i_1,\cdots, i_j\}$ from $\{1,\cdots n\}$ and, the third union is taken over  $[k_1,\cdots,k_j]_+$, the set of combinations of $j$   strictly positive integers $k_1,\cdots, k_j$, such that $\textstyle \sum_{\ell = 1}^j k_\ell = k$, that are assigned to $\{K_{i_1},\cdots, K_{i_j}\}$, thereby making the element $\{K_1,\cdots, K_n \}|_{K_\ell = 0, \ell \not\in \{i_1,\cdots,i_j\}}$.  
}
{Note that the first union runs no further than $n$ 
%$[k_1,\cdots,k_j]_+ = \emptyset$ for $j>n$ as  
as one cannot select more than $n$ different of the $Y_i$-s across the $k$ brackets in the product $\textstyle(Y_1+\cdots+Y_n)^k$.
}

\medskip

We can thus re-write \eqref{ais}
\begin{equation}
\left(\sum_{i = 1}^n Y_i \right)^k 
 = \sum_{j = 1}^k{\mathbf{1}_{(j\leq n)}}\sum_{\{i_1,\cdots,i_j\}\in \mathcal{I}_j}
\sum_{[k_1,\cdots,k_j]_+}
{k\choose k_1,\cdots, k_j}
\prod_{\ell =1}^j Y_{i_\ell}^{k_{\ell}}.
\label{ais2}
\end{equation}
In essence, thinking of the left-hand side as  $k$ brackets of the form $(Y_1+ \cdots + Y_n)$ that are multiplied together, we first pick the indices we want to appear in multiplying one term from each bracket, then look at the number of ways these indices can be selected from each of the $k$ brackets (this explains the appearance of the multinomial coefficient again).

\medskip

The fact that  $Y_1,\cdots,Y_n$ are iid means that 
\[
 {\rm E}\left[\prod_{\ell= 1}^j Y_{i_\ell}^{k_\ell}\right] =  \prod_{\ell=1}^j{\rm E}[Y^{k_i}],
\]
which is the same for multiple combinations of $\{i_1,\cdots, i_j\} $ and $\{k_1,\cdots, k_j\}$. Indeed, for a given $j\in\{1,\cdots, k\}$ and combination $[k_1,\cdots, k_j]_+$ there are 
%\[
%\frac{n!}{(n-j)!} \text{ choices of $\{i_1,\cdots, i_j\}$ that produce the factor } \prod_{\ell=1}^j{\rm E}[Y^{k_i}].
%\]
%More precisely,
%there are 
$
\textstyle{ {n \choose j }}$  ways of {picking $j\leq n$ indices} from $n$.
%and for a given set of indices $\{i_1,\cdots, i_j\}$ there are   $j!$ ways we can assign the powers $\{k_1,\cdots, k_j\}$ these to the indices.
 In conclusion \eqref{j!} holds.
\end{proof}

We conclude with an unusual combinatorial lemma that we could not find in the literature. 
\begin{lemma}\label{combinatoriclemma}We have
we have 
\[
(k+1)! = \sum_{j = 1}^k2^j \sum_{[k_1,\cdots, k_j]_+} \frac{1}{j!}{k \choose k_1, \cdots, k_j} \prod_{i =1}^j(k_i-1)!, \qquad k\in\mathbb{N}.
\]
\end{lemma}
\begin{proof}
Recall that, for $j,k\in \mathbb{N}$ such that $1\leq j\leq k$, the $(k,j)$-Bell polynomial is given  by
\[
B_{k,j}(x_1,\cdots, x_{k-j+1}): = 
\sum_{(\ell_1,\cdots, \ell_{k-j+1})}\frac{k!}{\ell_1!\cdots\ell_{k-j+1}!} \prod_{p= 1}^{k-j+1}\left(\frac{x_\ell}{p!}\right)^{\ell_p},
\]
where the sum is taken over all combinations of $(\ell_1,\cdots, \ell_{k-j+1})$ such that 
\begin{equation}
\sum_{i =1}^{k-j+1} \ell_i = j\,\text{ and }\,\sum_{i=1}^{k-j+1} i \ell_i = k.
\label{sums}
\end{equation}
One should think of $(\ell_1,\cdots,\ell_{k-j+1})$ as  representing the partition of $\{1,\cdots,k\}$ into $1\leq j\leq k$ non-empty sub blocks. The quantity $\ell_i$ is thus the number of blocks of size $i$ so that the first sum in \eqref{sums} indicates that the total number of blocks is $j$ and the total number of items in these blocks is $k$. 
Note that some of the values of $\ell_i$ may be zero. Note also that the index runs up to $k -j +1$ as the largest of the $j$ blocks can be obtained by having $j-1$ blocks of size 1 and 1 block of size $k-j+1$. 
As such, the $(k,j)$-Bell polynomial tells us the number of ways of generating $j$ non-empty sub blocks from $\{1,\cdots, k\}$ by setting $x_1 = \cdots = x_{k-j+1} = 1$.

\medskip

If a configuration $\{k_1,\cdots, k_j\}$ has block count $\{\ell_1,\cdots, \ell_{k-j+1}\}$, then %There is a one-to-one correspondence between each configuration of $(\ell_1,\cdots, \ell_{k-j+1})$ and of $[k_1,\cdots, k_j]_+$. 
%For corresponding pairs, 
%$\{\ell_1,\cdots, \ell_{k-j+1}\}$ and $\{k_1,\cdots, k_j\}$, we note that 
\[
\prod_{i = 1}^j(k_i -1)! = \prod_{p= 1}^{k-j+1} ((p-1)!)^{\ell_p}.
\]
For each $j=1,\cdots, k$, counting the functionals $\textstyle \prod_{i = 1}^j(k_i -1)!$ with its respective multiplicity is equivalent to counting the functionals $\textstyle \prod_{p= 1}^{k-j+1} ((p-1)!)^{\ell_p}$ with their respective multiplicity. 

\medskip

Considering again the multinomial formula \eqref{j!}, we  see that for each $j=1,\cdots, k$ and $\{k_1, \cdots, k_j\}\in[k_1,\cdots, k_j]_+$ the functional  $\textstyle \prod_{i = 1}^j(k_i -1)!$  has a respective multiplicity given by
\[
{\frac{1}{j!}}{k \choose k_1,\cdots, k_j}.
\]
Note, the factor $1/j!$ is present as we are not concerned about the order of the blocks. Similarly, for each $\{\ell_1,\cdots, \ell_{k-j+1}\}\in(\ell_1,\cdots, \ell_{k-j+1})$, the functional $\textstyle \prod_{p= 1}^{k-j+1} ((p-1)!)^{\ell_p}$ has multiplicity
\[
\frac{1}{\ell_1!\cdots\ell_{k-j+1}!}\times\frac{k!}{(1!)^{\ell_1}\cdots ((k-j+1)!)^{\ell_{k-j+1}}},
\]
where the second fraction is the combinatoric for selecting $\ell_1$ blocks of size 1, $\ell_2$ blocks of size 2.... and $\ell_{k-j+1}$ blocks of size $k-j+1$, and the first fraction factors out the multiplicity of order of blocks within each size category.

 %The mapping $[k_1,\cdots, k_j]_+\mapsto(\ell_1,\cdots, \ell_{k-j+1})$ is surjective. Moreover, for each $\{\ell_1,\cdots, \ell_{k-j+1}\}\in(\ell_1,\cdots, \ell_{k-j+1})$ the number of $\{k_1,\cdots, k_j\}\in[k_1,\cdots, k_j]_+$
%that maps to $\{\ell_1,\cdots, \ell_{k-j+1}\}$ is $(\ell_1!\cdots\ell_{k-j+1} !)/j!$ because

It follows that 
\begin{align}
&\sum_{j = 1}^k   2^j   \sum_{[k_1,\cdots,k_j]_+} {\frac{1}{j!}}{k \choose k_1,\cdots, k_j}\prod_{i = 1}^j(k_i -1)! \notag\\
&=\sum_{j = 1}^k 2^j  
\sum_{(\ell_1,\cdots, \ell_{k-j+1})} \frac{1}{\ell_1!\cdots\ell_{k-j+1} !}  \frac{k!}{\prod_{p= 1}^{k-j+1}(p!)^{\ell_p}} \prod_{p= 1}^{k-j+1}\left((p-1)!\right)^{\ell_p}\notag\\
&=\sum_{j = 1}^k 2^j  
\sum_{(\ell_1,\cdots, \ell_{k-j+1})}\frac{k!}{\ell_1!\cdots\ell_{k-j+1}! } \prod_{p= 1}^{k-j+1}\left(\frac{(p-1)!}{p!}\right)^{\ell_p}\notag\\
&=\sum_{j = 1}^k 2^j  B_{k,j}(0!,\cdots, (k-j)!).
\label{piece1}
\end{align}
Standard texts on combinatorics (see e.g \cite{Comtet}) identify 
\begin{equation}
B_{k,j}(0!,\cdots, (k-j)!) =  \left[\substack{\mbox{ $k$ }\\\ \\ \mbox{ $j$ }}\right],
\label{piece2}
\end{equation}
which is a Stirling number of the first kind. Moreover, for the latter, it is known that 
\begin{equation}
\sum_{j = 1}^k x^j   \left[\substack{\mbox{ $k$ }\\ \\\mbox{ $j$ }}\right] = x(x+1)\cdots (x + k - 1).
\label{piece3}
\end{equation}
Putting the pieces together from \eqref{piece1}, \eqref{piece2} and \eqref{piece3}, the desired result holds.
\end{proof}

\bibliography{references}{}
\bibliographystyle{plain}

\end{document}